\documentclass[10pt]{amsart}
\usepackage{amssymb,latexsym, amscd,pb-diagram}
\usepackage{sseq}
\usepackage[all]{xy}
\usepackage[square, numbers]{natbib}
\usepackage{graphicx}
\usepackage{mathrsfs}
\usepackage{color}
\usepackage{amsmath}
\usepackage{multirow}
\usepackage{stmaryrd}
\usepackage{hyperref} 
\usepackage{todonotes}

\usepackage{tikz}
\usepackage{dsfont}
\usetikzlibrary{matrix}

 \newtheorem{theorem}{Theorem}[section]

 \newtheorem{lemma}[theorem]{Lemma}
 \newtheorem{proposition}[theorem]{Proposition}
 \theoremstyle{definition}
 \newtheorem{definition}[theorem]{Definition}
 \theoremstyle{definition}
 \newtheorem{example}[theorem]{Example}

 \numberwithin{equation}{section}

\newcommand{\ben}{\begin{equation}}
\newcommand{\een}{\end{equation}}

\newcommand{\integer}{\ensuremath{{\mathbb Z}}}

\newcommand{\real}{\ensuremath{{\mathbb R}}}

\newcommand{\DD}{{\mathcal D}}

\newcommand{\id}{{\text{id}}}

\newcommand{\FF}{{\mathcal F}}

\newcommand{\GG}{{\mathcal G}}
\newcommand{\CC}{\mathcal{C}}

\newcommand{\MM}{\mathcal{M}}

\newcommand{\Hom}{\mathrm{Hom}}
\newcommand{\Ext}{\mathrm{Ext}}

\newcommand{\target}{\mathsf{t}}
\newcommand{\source}{\mathsf{s}}
\newcommand{\ident}{\mathsf{e}}

\newcommand{\twoarrows}{\rightrightarrows}

\newcommand{\To}{\longrightarrow}

\newcommand{\Groupoid}{{\mathfrak{TopGpds}}}
\newcommand{\Gerbe}{{\mathfrak{LftGrbs}}}

\newcommand{\Map}{\ensuremath{{\mathrm{Map}}}}

\newcommand{\Vect}{\ensuremath{{\mathrm{Vect}}}}

\newcommand{\Ker}{\mbox{\rm Ker\,}}
\newcommand{\Coker}{\mbox{\rm Coker\,}}

\newcommand{\Tot}{\mbox{\rm Tot\,}}

\newcommand\Rep{\operatorname{Rep}}

\newcommand{\End}{\operatorname{End}}

\newcommand{\im}{\mathop{\rm Im\,}}

\renewcommand{\Im}{\im}

\newcommand{\Topab}{\ensuremath{\mathfrak{Topab}}}
\newcommand{\SMGrbs}{\ensuremath{\mathfrak{SM}\text{-}\mathfrak{Grbs}}}

\newcounter{commentcounter}

\newcommand{\commentb}[1]
{\stepcounter{commentcounter}
  \textbf{Comment \arabic{commentcounter} (by Bernardo)}: 
{\textcolor{blue}{#1}} }

\begin{document}

\title[Pontrjagin duality on multiplicative Gerbes]{Pontrjagin duality on multiplicative Gerbes}

\thanks{Blanco acknowledges the support of Minciencias through `Convocatoria 785 - Convocatoria de Doctorados Nacionales de 2017'.
Uribe acknowledges the financial support of  Conacyt
through grant number CB-2017-2018-A1-S-30345-F-3125. This project has been carried out within a Research group linkage programme Greifswald-Barranquilla-Medell\'in sponsored by the Alexander Von Humboldt Foundation.}
\author{Jaider Blanco}
\address{Departamento de Matem\'{a}ticas y Estad\'istica, Universidad del Norte, Km.5 V\'ia Antigua a Puerto Colombia, 
Barranquilla, Colombia.}
\email{jaiderb@uninorte.edu.co}
\author{Bernardo Uribe}
\address{Departamento de Matem\'{a}ticas y Estad\'istica, Universidad del Norte, Km.5 V\'ia Antigua a Puerto Colombia, 
Barranquilla, Colombia.}
\address{Max Planck Institut f\"ur Mathematik\\
		Vivatsgasse 7\\ 
		Bonn 53115, Germany}
	\email{bjongbloed@uninorte.edu.co, uribe@mpim-bonn.mpg.de}

\author{Konrad Waldorf}
\address{Universit\"at Greifswald, Institut f\"ur Mathematik und Informatik, Walther-Rathenau Str. 47, 17487 Greifswald, Germany.}
\email{konrad.waldorf@uni-greifswald.de }

\begin{abstract}
We use Segal-Mitchison's cohomology of topological groups
to define a convenient model for topological gerbes. We introduce multiplicative gerbes
over topological groups in this setup and we define its representations. For a specific choice of representation,
we construct its category of endomorphisms and we show that it induces a new multiplicative gerbe over another
topological group. This new induced group is fibrewise Pontrjagin dual to the original one and therefore
we  called the pair of multiplicative gerbes `Pontrjagin dual'. We show that
Pontrjagin dual multiplicative gerbes have equivalent categories of representations and moreover, 
we show that their monoidal centers are equivalent. Examples of Pontrjagin dual multiplicative gerbes
over finite and discrete, as well as compact and non-compact Lie groups are provided.

\end{abstract}

\maketitle

\section*{Introduction} 
Gerbes with structure group $U(1)$ over a manifold are geometrical objects whose isomorphism classes are classified
by degree three integer cohomology classes \cite{Brylinski}.  They could be thought of as groupoid extensions of the underlying
manifold by the groupoid $U(1)[1]$ \cite{Laurent-Gengoux}, as $U(1)$-bundle gerbes \cite{Murray-1996}, or as principal $U(1)[1]$-2-bundles
 \cite{Waldorf-4}. In each of these versions,  gerbes over manifolds define a monoidal 2-category. Thus, one may consider
weak monoid objects \cite{BaezLauda} in the 2-category of gerbes. These weak monoid objects in gerbes are what are known in the literature as multiplicative gerbes \cite{brylinski4,Carey_Bai-Ling, Waldorf-multiplicative}, and of
particular interest has been the study of multiplicative gerbes over compact, simple and simply connected Lie groups \cite{Meinrenken, gawedzki2, Carey_Bai-Ling, Waldorf-multiplicative,gawedzki9} due, among others, to its relation with Chern-Simons theories
and Wess-Zumino-Witten models. The classification of isomorphism classes of multiplicative gerbes
over compact Lie groups is well understood; isomorphic multiplicative gerbes over a fixed compact Lie group 
are classified by degree four integer cohomology classes of its classifying space, and every known equivalence relation between  multiplicative
gerbes over compact Lie groups involves isomorphic underlying Lie groups.
In the paper we address the question how classify 
multiplicative gerbes up to Morita equivalence. Namely, may multiplicative
gerbes over non-isomorphic Lie groups have equivalent categories of representations? Let us elaborate.

Recall that multiplicative gerbes are weak monoid objects in the 2-category of gerbes, and monoid
objects in a 2-category may act on objects of the same 2-category. Therefore
we may define the 2-category of representations of a multiplicative gerbe as the 2-category of its weak module objects \cite{BaezLauda}.  Now the previous questions make sense; in particular, we will call multiplicative gerbes \emph{Morita equivalent} if they have equivalent 2-categories of representations.
Clearly,  isomorphic multiplicative gerbes over isomorphic Lie groups will be Morita equivalent. In this work, we are looking for Morita equivalent multiplicative gerbes defined over non-isomorphic Lie groups.

When considering multiplicative gerbes over finite groups, the previous question has been completely solved
\cite{Naidu, Uribe}. A multiplicative gerbe over a finite group defines a pointed fusion category \cite{ENO}
and its module categories (representations) are well understood \cite{Ost, Ost-2, Naidu}. The endomorphism
category of particular module categories induce again pointed fusion categories of the same Frobenius-Perron dimension, and in these cases
the pointed fusion categories become Morita equivalent \cite{Naidu}. An explicit description
of the Morita equivalence classes in terms of groups and cocycles has been obtained by the second author in \cite{Uribe}.

This work generalizes the ideas and procedures carried out to classify Morita equivalence classes of
multiplicative gerbes over finite groups to the more general case of topological groups. To make this generalization
possible we construct a model for gerbes over topological spaces similar to the one done for the finite group case.
We draw on the cohomology theory for topological groups constructed by Segal and Mitchison \cite{Segal-Cohomology},
and we define our model for gerbes using spaces and degree 2 cocycles over the spaces. We call these gerbes
\emph{Segal-Mitchison gerbes}.

We emphasize here that the cohomology theory of topological groups defined by Segal and Mitchison incorporates
the information about the topological invariants of these groups, as well as the cohomological
information of the complex of continuous cochains of the group. This cohomology theory
permits to treat in equal footing finite, discrete, compact and non-compact Lie groups.

We define the monoidal 2-category of Segal-Mitchison gerbes, we consider weak monoid objects in this category
as our model for multiplicative gerbes, and we define their 2-categories of representations. 
We take a representation of a multiplicative gerbe satisfying a cohomological condition  expressed with the help of the Lyndon-Hochschild-Serre spectral sequence, and we consider the monoidal category of its endomorphism.
We show that this monoidal category of endomorphisms defines a new multiplicative Segal-Mitchison gerbe.
This new multiplicative Segal-Mitchison gerbe is called the `fibrewise Pontrjagin dual' of the original one, because
it is constructed throughout a dualization procedure over a normal and  abelian subgroup of the  original multiplicative gerbe.
This dualization produces a new group with a normal and abelian group which is the Pontrjagin dual of the 
normal and abelian group of the original group.

We note here that the monoidal category of endomorphisms is  a gerbe of a particular kind called
\emph{Lifting bundle gerbe} \cite{Murray-1996}. We include an Appendix where we summarize the properties of lifting bundle gerbes, and we show that  Segal-Mitchison gerbes and lifting bundle gerbes form canonically equivalent monoidal 2-categories.

We present an array of examples of Pontrjagin dual multiplicative gerbes, including finite, discrete, compact and non-compact Lie
groups, to demonstrate that the model of Segal-Mitchison gerbes is ample enough to encompass them all. It is important
to notice that the classical Pontrjagin duality on abelian topological groups could be recovered as the fibrewise Pontrjagin
dual of the trivial multiplicative Segal-Mitchison gerbe defined by the abelian groups once they fiber over a trivial group.

Then we proceed to show that the 2-categories of representations of fibrewise Pontrjagin dual multiplicative
Segal-Mitchison gerbes are equivalent, and therefore justify calling them `Morita equivalent'. We show furthermore that this equivalence induces an equivalence of the categorical centers 
 of the fibrewise Pontrjagin dual multiplicative Segal-Mitchison gerbes.

We have thus succeeded in finding sufficient conditions for two multiplicative Segal-Mitchison gerbes to be Morita equivalent.
 We conjecture that these conditions are also necessary, as it is the case for multiplicative gerbes over finite groups.
Unfortunately, the foundations underlying the categorial constructions performed in this work with topological
groups are not as developed as the ones over fusion categories. These foundations would have to be first developed 
in order to be able to generalize the procedures and proofs that have been performed using pointed fusion categories.
Thus, the classification of Morita equivalence classes of multiplicative Segal-Mitchison gerbes remains open. Nevertheless, the results of  this work 
offer a clear path on how to carry out the complete classification.

We start this work in \S \ref{section cohomology} with and introduction to Segal-Mitchison's cohomology of topological groups; we summarize its most important features, we show its relation to the cohomology of continuous cochains and
to the cohomology of the classifying space of the group, and we introduce the Lyndon-Hochschild-Serre spectral sequence for this cohomology theory. In 
\S \ref{section SM gerbes} we define the monoidal 2-category of Segal-Mitchison gerbes, in \S \ref{section multiplicative}
we define multiplicative Segal-Mitchison gerbes as weak monoid objects in Segal-Mitchison gerbes, and
in \S \ref{section representations} we define the 2-category of representations of multiplicative Segal-Mitchison gerbes as weak module objects over weak monoid objects in Segal-Mitchison gerbes.
In \S \ref{section pontrjagin} we introduce the fibrewise Pontrjagin duality of multiplicative Segal-Mitchison gerbes
and we present several examples of this duality. Finally in \S \ref{section Morita} we show that fibrewise Pontrjagin dual
multiplicative Segal-Mitchison gerbes have equivalent 2-categories of representations, thus becoming Morita equivalent.

The final two sections include a summary on crossed modules and 2-groups in \S \ref{Appendix} (Appendix A) and
a summary on Lifting bundle gerbes and their relation to Segal-Mitchison gerbes in \S \ref{Appendix B} (Appendix B).

\section{Segal-Mitchison cohomology} \label{section cohomology}

G. Segal in \cite{Segal-Cohomology} defined cohomology groups suited to understand
the cohomology of a topological group with coefficients in a  topological abelian group.
The description of this cohomology theory is closely related to the cohomology of a
discrete group with coefficients in an abelian group.
In what follows we will review the definition of Segal that appears in \cite{Segal-Cohomology}.

\subsection{Coefficients} The coefficients of the Segal-Mitchison cohomology are in the category $\Topab$ of compactly generated and
locally contractible Hausdorff topological abelian groups and continuous homomorphisms. A sequence
$$A' \stackrel{i}{\To} A \stackrel{p}{\To} A''$$
in $\Topab$ is called a short exact sequence if $i$ embeds $A'$ as a closed subgroup in $A$, $p$ induces
an isomorphism of topological groups $A/A' \cong  A''$, and $p: A \to A''$ is a principal $A'$-bundle.
If $A \in \Topab$, then $EA$ and $BA:= EA / A$ are also in $\Topab$ and
$$A \to EA \to BA$$
is a short exact sequence. Recall that $EA$ is a contractible topological space obtained from the nerve of the 
contractible category 
whose space of objects is $A$ and whose spaces of morphisms is $A \times A$, one morphism for each pair of elements in $A$.
Since $A$ is abelian, the contractible space $EA$ may be endowed with the structure of a topological abelian group \cite[Appendix A]{Segal-Cohomology}.

The previous construction may be iterated, thus giving the short exact sequences in $\Topab$
$$B^lA \stackrel{i_l}{\To} EB^lA \stackrel{p_l}{\To} B^{l+1}A$$
for every $l \geq 0$. Composing the projections $EB^lA \stackrel{p_l}{\To} B^{l+1}A$ with the embeddings $B^{l+1}A \stackrel{i_{l+1}}{\To} EB^{l+1}A$
we obtain continuous homomorphisms $ EB^lA \stackrel{\partial_l}{\To} EB^{l+1}A$ which can be assembled to produce
a resolution of $A$ by contractible groups in $\Topab$
$$A \hookrightarrow EA \stackrel{\partial_0}{\To} EBA \stackrel{\partial_1}{\To} EB^2A\stackrel{\partial_2}{\To} EB^3A \stackrel{\partial_3}{\To} \cdots.$$
For simplicity, the subindex in $\partial_i$ will be omitted in what follows.

\subsection{Cohomology of spaces}
Let $X$ be a paracompact space and consider the complex
$$C^l(X, \underline{A}) : = \Map(X, EB^lA)$$
with differential $\partial f:= \partial \circ f$.
The space of maps is endowed with the compact-open topology and has a canonical group structure. Hence, 
$C^l(X, \underline{A})$ belongs to $\Topab$.

\begin{definition} Let $X$ be a paracompact space.
The cohomology of the complex $(C^*(X, \underline{A}), \partial)$ is called the {\it{Segal-Mitchison cohomology}}
of $X$ with coefficients in $A$ and will be denoted by
$$H^*(X, \underline{A}) := H^* (C^*(X, \underline{A}), \partial).$$
\end{definition}
Note that the cocycles $Z^l(X, \underline{A})$ consist of maps $X \stackrel{f}{\to} EB^lA$ such that the composition
$X \stackrel{f}{\to} EB^lA \stackrel{p_l}{\to} B^{l+1}A$ is the identity in $B^{l+1}A$. Since the fiber of $p_l$ at the identity element is
$B^lA$, we have an isomorphism of topological groups:
$$Z^l(X, \underline{A}) \cong \Map(X, B^lA).$$
By the classification of principal bundles over paracompact spaces, we know that a map $X \to B^lA$ is null-homotopic
if and only if the map lifts to $EB^{l-1}A$. Therefore we obtain the isomorphism
$$H^l(X,\underline{A}) \cong [X,B^lA]$$
whenever $l \geq 1$. For $l=0$ we obtain the isomorphism
$$H^0(X,\underline{A}) \cong \Map(X,A)$$
and we may say that the Segal-Mitchison cohomology calculates the cohomology of $X$ with coefficients in the sheaf 
$\underline{A}$ of $A$-valued functions, i.e. is the derived functor of $\Map(X,A)$.

Whenever we have a short exact sequence in $\Topab$
$$A' \stackrel{i}{\To} A \stackrel{p}{\To} A'',$$
a standard diagram chase argument shows that there is an induced long exact sequence in cohomology
$$ \to H^{k-1}(X,\underline{A''}) \to H^k(X,\underline{A'}) \to H^k(X,\underline{A}) \to H^k(X,\underline{A''}) \to H^{k+1}(X,\underline{A'}) \to.$$

\subsection{Cohomology of simplicial spaces}
Whenever $X_\bullet$ is a simplicial paracompact space we may define the double complex
$$C^{p,q}(X_\bullet, \underline{A}): = \Map(X_p, EB^qA)$$
whose horizontal differential 
$$\delta : C^{p,q}(X_\bullet, \underline{A}) \to C^{p+1,q}(X_\bullet, \underline{A})$$
is the one induced by the structural maps $X_{p+1} \to X_p$ of the simplicial space, and whose vertical differential is $\partial$.
The total complex of the double complex will be denoted by 
$$C^{*}(X_\bullet, \underline{A}) := \Tot(C^{*,*}(X_\bullet, \underline{A}))$$
and its differential is $d:=\delta + (-1)^p\partial$.

\begin{definition} Let $X_\bullet$ be a simplicial paracompact space. The cohomology of the complex  
  $C^{*}(X_\bullet, \underline{A})$ with differential $d$ is denoted by 
$$H^*(X_\bullet, \underline{A}) := H^* \left(C^{*}(X_\bullet, \underline{A}), d \right),$$ and
is called the Segal-Mitchison cohomology of the simplicial space $X_\bullet$ with coefficients in the topological abelian group $A$.
\end{definition}
We may filter the total complex by the subcomplexes $\Tot(C^{* \geq p,*}(X_\bullet, \underline{A}))$ thus
obtaining a spectral sequence whose first page is isomorphic to
$$E_1^{p,q}= H^q(X_p, \underline{A}).$$
Note that whenever the simplicial space $X_\bullet$ is homotopy equivalent to a simplicial set, or when $A$ is contractible, we know that  $E_1^{p,q}=0$ for $q >0$. For $q=0$ we obtain $E_1^{p,0} = \Map(X_p,A)$, and therefore
in this case $H^*(X_\bullet, \underline{A})$ is simply the cohomology of the complex $(\Map(X_*,A), \delta)$. The cohomology 
of this complex will be denoted
$$H^*_{cont}(X_\bullet, A) := H^*(\Map(X_*,A), \delta)$$
and will be called the {\it {cohomology of continuous cochains}}.

\begin{lemma} \label{lemma A contractible}Suppose that $X_\bullet$ is a simplicial paracompact space and $A$ is a contractible topological group, or alternatively suppose that the simplicial space $X_\bullet$ is homotopy equivalent to 
a simplicial set. Then, in either case, the forgetful homomorphism 
induces an isomorphism of cohomologies:
$$H^*(X_\bullet, \underline{A}) \stackrel{\cong}{\to} H^*_{cont}(X_\bullet, A).$$
\end{lemma}

\subsection{Cohomology of topological groups}  \label{subsection cohomology topological groups}
Let $G$ be a paracompact and compactly generated topological group. Denote by $G_\bullet$ the simplicial space
defined by the nerve of the topological category with one object and $G$ for the space of morphisms. 
In this case we have that $G_p=G^p$ and the differential 
$$ \delta_G : \Map(G^{p},EB^qA) \to \Map(G^{p+1},EB^qA)$$
is defined by the  formula:
\begin{align*}\delta_G&(f)(h_0,...,h_p)= \\
& f(h_1,...,h_p)+\sum_{i=1}^{p}(-1)^if(h_1,...,h_{i-1}h_{i},...,h_p) +(-1)^{p+1}f(h_0,...,h_{p-1}).
\end{align*}

\begin{definition}
Let $G$ be a paracompact and compactly generated topological group and let $A$  a compactly generated and locally contractible Hausdorff topological abelian group. The Segal-Mitchison cohomology of the group $G$
with coefficients in  $A$ is:
$$H^*(G_\bullet, \underline{A}) = H^*(\Tot(C^{*,*}(G_\bullet, \underline{A}), d_G=\delta_G+ (-1)^p\partial).$$
\end{definition}

Segal in \cite[\S 4]{Segal-Cohomology} shows
that the cohomology groups $H^*(G_\bullet, \underline{A})$ for $*=1,2$ have their usual interpretations. Let us recall the proofs since they will be used in what follows.
\begin{lemma} The first Segal-Mitchison cohomology group
 is isomorphic to the topological group of continuous homomorphisms, i.e.
$$H^1(G_\bullet, \underline{A}) \cong \Hom(G,A).$$
\end{lemma}
\begin{proof}
An element in $Z^1(G_\bullet, \underline{A})$ consists of a pair of maps $f_1:G \to EA$ and $f_0:\{*\} \to EBA$ such
that $\delta_Gf_1(g_1,g_2)=f(g_2)-f(g_1g_2)+f(g_1)=0$,  $\partial f_1(g)=- \delta f_0(*)= 0$ and $\partial f_0(*)=0.$ The last two equations
imply that the image of $f_1$ lies in $A$, the kernel of $\partial : EA \to EBA$, and the image of $f_0$ in $BA$, the kernel of $\partial : EBA \to EB^2A$. The first equation implies that $f_1: G \to A \subset EA$ is a 
homomorphism of groups. Therefore we obtain the isomorphism of groups
$$Z^1(G_\bullet, \underline{A})\cong \Hom(G,A) \oplus BA.$$
 Now, the image of $C^0(G_\bullet, \underline{A})=\Map({*},EA)$ under $d=\delta + \partial$ is simply $\{0\}\oplus BA$ and therefore
 we obtain the desired isomorphism $H^1(G_\bullet, \underline{A}) \cong \Hom(G,A)$.
 \end{proof}
 \begin{lemma} \label{lemma classification of central extensions}
The second Segal-Mitchison cohomology group
 is isomorphic to the group of isomorphism classes of $A$-central extensions of $G$, i.e.
$$H^2(G_\bullet, \underline{A}) \cong \Ext(G,A).$$
\end{lemma}

\begin{proof}
We will consider $A$-central extensions of groups, $A \to P \stackrel{p}{\to} G$, i.e. short exact sequences of groups where $A$ is central in $P$. We recall that short exact sequence means, in particular, that $P$ is a principal $A$-bundle over $G$.  Two of these $A$-central extensions $P \stackrel{p}{\to} G$ and $P' \stackrel{p'}{\to} G$ are isomorphic if there
is a homomorphism of groups $\phi : P \to P'$ with $p' \circ \pi=p$ such that $\phi$ is a morphism of principal $A$-bundles.
Denote by $\Ext(G,A)$ the set of isomorphism classes of $A$-central extensions of $G$.

The group of 2-cocycles $Z^2(G_\bullet, \underline{A})$ consist of three maps $F_2: G\times G \to EA$, $F_1: G \to EBA$ and $F_0: \{*\} \to EB^2A$
satisfying the following equations:
\begin{align*}
F_2(g_2,g_3)+F_2(g_1,g_2g_3)&=  F_2(g_1g_2,g_3) + F_2(g_1,g_2) \\
\partial F_2(g_1,g_2) & = F_1(g_1g_2)- F_1(g_1) -F_1(g_2) \\
\partial F_1(g)& = F_0(*)-F_0(*)=0\\
\partial F_0(*)&= 0.
\end{align*}
Hence $F_1$ takes values in $BA \subset EBA$ and we may take the pullback bundle $F_1^*EA$ over $G$ as
$F_1^*EA : =\{(g,b) \in G \times EA | F_1(g)= p_0(b) \} $. We may define the group structure on $F_1^*EA$ by the equation
$$(g_1,b_1) \cdot (g_2, b_2) : = (g_1g_2, b_1+b_2+ F_2(g_1,g_2))$$
since we know by the second equation that 
$$\partial ( b_1+b_2+ F_2(g_1,g_2)) =  \partial(b_1) + \partial (b_2) + F_1(g_1g_2)-F_1(g_1) - F_1(g_2)= F_1(g_1g_2).$$ 
The associativity of the product structure in $F_1^*EA$ follows from the first equation thus making $F_1^*EA$ an $A$-central
extension of $G$. Therefore we have constructed an $A$-central extension $F_1^*EA$ over $G$ with the information of the maps
$F_1$ and $F_2$ (the information of $F_0$ is not important).

Now for an $A$-central extension $P \to G$, which is moreover a principal $A$-bundle, we may define its classifying map $F_1 : G \to BA$
together with its lift $F_1' : P \to EA$, a morphism of principal $A$-bundles. For any pair of elements $p_1, p_2 \in P$
we define $$\overline{F}_2(p_1,p_2) := F_1'(p_1p_2) - F_1'(p_2) - F_1'(p_1)$$ and we note that for any $a_1,a_2 \in A$
we have $\overline{F}_2(p_1,p_2) =\overline{F}_2(a_1p_1,a_2p_2)$. Therefore $\overline{F}_2$ is $A \times A$-invariant and therefore it induces a map
$F_2: (A \backslash P) \times (A \backslash P) \cong G \times G \to EA$. The associativity of $P$ implies that $\delta_G F_2=0$ and the definition of $\overline{F}_2$
implies that $\partial F_2(g_1,g_2) \cong F_1(g_1g_2)- F_1(g_1) -F_1(g_2)$. Therefore the maps $F_2$ and $F_1$ are the essential information
that defines a 2-cocycle in $Z^2(G_\bullet, \underline{A})$.
 
The fact that cohomologous 2-cocycles induce isomorphic $A$-central extensions and that isomorphic $A$-central extension induce cohomologous 2-cocycles
are straightforward.
\end{proof}

 By Lemma \ref{lemma A contractible}, whenever $A$ is contractible we obtain an isomorphism with the cohomology of continuous cochains of the group with coefficients in $A$
$$H^*(G_\bullet, \underline{A}) \cong H^*_{cont}(G,A).$$
Whenever $A$ is discrete, by \cite[Prop. 3.3]{Segal-Cohomology} there is an isomorphism with the cohomology of the complex of singular cochains with coefficients in $A$
$$H^*(G_\bullet, \underline{A}) \cong H^*(BG,A).$$

The short exact sequence $A \to EA \to BA$ induces a long exact sequence in cohomology groups, and  since $EA$ is a contractible, we obtain the long exact sequence in cohomology
$$\to H^n_c(G, EA) \to H^n(G_\bullet, \underline{BA}) \to H^{n+1}(G_\bullet, \underline{A}) \to H^{n+1}_c(G,EA) \to.$$
 When $n=2$ the homomorphism $H^2(G_\bullet, \underline{BA}) \to H^3(G_\bullet, \underline{A})$
could be understood as the assignment that takes the isomorphism class of a $BA$-central extension of $G$,
$$BA \To \Gamma \to G,$$
to the isomorphism class of the crossed module $EA \to \Gamma$ which induces the exact sequence
$$A \To EA \To \Gamma \to G.$$
The
2-group $\Gamma \times EA \twoarrows \Gamma$ that defines the crossed module could be understood as a 2-group central extension of the 2-group $ G \twoarrows G$ by the 2-group $A \twoarrows \{*\}$ and its extension class is an element in the
cohomology group
$H^3(G_\bullet, \underline{A})$.
We refer to the Appendix in section \S \ref{subsection classification crossed modules} where we elaborate on the previous statement.

\subsection{Segal-Mitchison cohomology as a derived functor}The cohomology groups $H^*(G_\bullet, \underline{A})$ could be understood as the derived functor of $G$-invariant sections
of the trivial $G$-equivariant topological abelian group $A$. Whenever $A$ is endowed with an action of $G$ we need
to elaborate a little further. 

Denote by $G$-$\Topab$ the category of compactly generated locally contractible topological
abelian groups on which $G$ acts continuously \cite[\S 2]{Segal-Cohomology}. A sequence $A' \stackrel{}{\to} A \stackrel{}{\to} A''$ in $G$-$\Topab$ 
is called a short exact sequence if it is one in $\Topab$ when the action of $G$ is neglected. Call
an object in $G$-$\Topab$ {\it{soft}} if it is of the form $\Map(G,A)$ where $A$ is a contractible group and $\Map$ denotes
space of continuous maps endowed with the compact open topology \cite[pp. 380]{Segal-Cohomology}. Note that a soft group is  contractible.
Segal shows that every object $A$ in $G$-$\Topab$ has a soft resolution
$$A \to E_{(G)}A \to E_{(G)}B_{(G)}A \to E_{(G)}B^2_{(G)}A \to \cdots$$
where $E_{(G)}A :=\Map(G,EA)$ and $B_{(G)}A$ fits in the short exact sequence
 \begin{align}
 A \To E_{(G)}A \To B_{(G)}A \label{short exact sequence A->E_GA -> B_GA}
\end{align}
where $A$ embeds in $E_{(G)}A$ as the constant maps taking values in $A$.

Alternatively one may take the double complex $\Map(G^{k+1},EB^lA)$, $k ,l \geq 0$, with differentials
$\partial$ and $\delta$, with $\delta$ defined by the following equation:
$$\delta(f)(h_1,...,h_n)= f(h_2,...,h_n)+\sum_{i=1}^{n-1}(-1)^if(h_1,...,h_ih_{i+1},...,h_n).$$
The total complex becomes a soft resolution of $A$ in $G$-$\Topab$.
Then, the derived functor of the $G$-invariant sections $\Gamma^G(A)$ of $A$ can be calculated by the cohomology
of the $G$-invariant sections of the previous total complex. The $G$-invariant double complex
is isomorphic to the double complex  $\Map(G^{k},EB^lA)$, $k ,l \geq 0$ whose differentials are $\partial$ and
\begin{align*}\delta_G&(f)(h_0,...,h_n)= \\
& f(h_1,...,h_n)+\sum_{i=1}^{n}(-1)^if(h_1,...,h_{i-1}h_{i},...,h_n) +(-1)^{n+1}f(h_0,...,h_{n-1})\cdot h_n.
\end{align*}
Denoting by $H^*(G_\bullet,\underline{A})$ the cohomology of 
$$(\Tot^*(\Map(G^{*},EB^*A)), d_G=\delta_G +(-1)^p\partial)$$
we see that $R^p\Gamma^G(A)= H^p(G_\bullet,\underline{A})$. 
Note that whenever the action of $G$ on $A$ is trivial,
the cohomology groups previously defined agree with the cohomology of topological groups defined in section \S \ref{subsection cohomology topological groups}.

 \subsection{Lyndon-Hochschild spectral sequence} \label{section LHSS}
 
 Let us consider a short exact sequence of locally contractible  and compactly generated groups
 $$S \hookrightarrow G \to K$$
 where 
$S$ is abelian and central.
Denote by $K \rtimes G$ the action groupoid determined by the right action of $G$ on $K$; the space
of objects is $K$, the space of morphisms is $K \times G$, $\source(k,g)=k$ and $\target(k,g)=kg$. Recall that the inclusion
$S[1] \to K \rtimes G$, $s \mapsto (1_K,s)$ is a weak equivalence of topological groupoids.

Consider first the double complex
$$C^{l,m}((K \rtimes G)_\bullet, \underline{A}):= \Map(K \times G^l, EB^mA)$$
with differentials $\delta_{K \rtimes G} : C^{l-1,m}((K \rtimes G)_\bullet, \underline{A}) \to C^{l,m}((K \rtimes G)_\bullet, \underline{A})$
\begin{align*}
&(\delta_{K \rtimes G} f) (k,g_1, \dots ,g_l) \\
&= f(kg_1, g_2, \dots, g_l) + \sum_{i=1}^{l-1} (-1)^i f(k, g_1, \dots, g_ig_{i+1}, \dots, g_l) + (-1)^lf(k,g_1, \dots, g_{l-1})
\end{align*}
and $\partial_2 : C^{l,m-1}((K \rtimes G)_\bullet, \underline{A}) \to C^{l,m}((K \rtimes G)_\bullet, \underline{A})$, $\partial_2 f = \partial \circ f$ (the reason for the sub-index 2 in the second differential will become clear in what follows). The cohomology of the total complex is $H^*((K \rtimes G)_\bullet, \underline{A})$.

Note that the groups $C^{l,m}((K \rtimes G)_\bullet, \underline{A})$ belong to $\Topab$ and they come endowed with
a continuous right $K$ action; for $f \in  C^{l,m}((K \rtimes G)_\bullet, \underline{A})$ and $h \in K$ let
$$(f \cdot h)(k,g_1, \cdots, g_l) := f(hk, g_1, \cdots, g_l).$$
The groups $C^{l,m}((K \rtimes G)_\bullet, \underline{A})$ then belong to $K$-$\Topab$.
Moreover,
they are {\it soft} since they are of the form $\Map(K, \Map(G^l,EB^mA))$ with $\Map(G^l,EB^mA)$ contractible. 

Let us now consider the quadruple complex 
$$C^{i,j,l,m} := \Map\left(K^i, EB^j\left(\Map(K \times G^l, EB^mA) \right) \right)$$  
with differentials
$\delta_K : C^{i-1,j,l,m} \to C^{i,j,l,m}$ defined by 
\begin{align*}
&(\delta_K F) (k_1, \cdots , k_i) \\
&= F(k_2, \cdots, k_i) + \sum_{r=1}^{i-1} (-1)^rF(k_1, \cdots, k_rk_{r+1}, \cdots, k_i) + (-1)^{i}F(k_1, \cdots, k_{i-1})\cdot k_i,
\end{align*}
$\partial_1 : C^{i,j-1,l,m} \to C^{i,j,l,       m}$, $\partial_1 F = \partial \circ F$ with
$$\partial : EB^{j-1}\left(\Map(K \times G^l, EB^mA)\right) \to EB^j\left(\Map(K \times G^l, EB^mA)\right),$$
and the ones induced by $\delta_{K \rtimes G}$ and $\partial_2$ which will be denoted with the same symbols
$\delta_{K \rtimes G} : C^{i,j,l-1,m} \to C^{i,j,l,m}$ and $\partial_2 : C^{i,j,l,m-1} \to C^{i,j,l,m}$.

At this point it is important to notice that the functors $E$ and $B$ send exact sequences in $K$-$\Topab$ to exact sequences
in $K$-$\Topab$ \cite[Prop. A.3]{Segal-Cohomology}. Therefore, if we have a complex of groups $(V^*, d)$ in $K$-$\Topab$, there is a
canonical isomorphism $H^*(EB^l V^*, EB^ld) \cong EB^l H^*(V^*,d)$ in $K$-$\Topab$. This canonical isomorphism will be
used in what follows.

Define the double complex
\begin{align} \label{definition D**}
D^{p,q} = \bigoplus_{{i+j=p}, {l+m=q}}  C^{i,j,l,m}
\end{align}
with differentials $d_K: D^{p-1,q} \to D^{p,q}$, $d_K:=\delta_K +(-1)^i \partial_1$, and $d_{K \rtimes G} : D^{p,q-1} \to D^{p,q}$,
$d_{K \rtimes G} :=\delta_{K \rtimes G} +(-1)^{l}\partial_2$, and consider the total complex
$\Tot(C^{*,*,*,*})=\Tot(D^{*,*})$ with total differential 
\begin{equation*}
d= d_K +(-1)^{p}d_{K \rtimes G} =\partial_K +(-1)^i \partial_1 +(-1)^{i+j}\delta_{K \rtimes G} +(-1)^{i+j+l}\partial_2. 
\end{equation*}
We claim
the following result.

\begin{proposition} Consider the short exact sequence of  locally contractible and compactly generated groups
 $$S \hookrightarrow G \to K$$
 where $K= S \backslash G$, $S$ is abelian and is central.
Then the inclusion of the $K$-invariant groups
$$C^*((K \rtimes G)_\bullet, \underline{A})^K \hookrightarrow \Tot(C^{*,*,*,*})$$
is a quasi-isomorphism. Hence the cohomology groups
$$H^* (G_\bullet, \underline{A}) \stackrel{\cong}{\To} H^*(\Tot(C^{*,*,*,*}),d)$$
are canonically isomorphic.
\end{proposition}
\begin{proof}
Filter the total complex $\Tot(C^{*,*,*,*})$ by the degree $q$ in $D^{p,q}$, i.e.
$F_q = \Tot(D^{*, *\geq q})$. The associated graded group $\bigoplus_{q \geq 0} F_q/F_{q+1} $ is isomorphic
to the total complex but where we only take into account the differential $d_K$. Since the groups
$\Map(K \times G^l, EB^mA) $ are soft in $K$-$\Topab$, the first page of the spectral sequence becomes
the $K$-invariant subgroup 
$$E_1^{p,0} \cong C^p((K \rtimes G)_\bullet, \underline{A})^K$$
with $E_1^{*,q \geq 1}=0$. Now, since $C^p((K \rtimes G)_\bullet, \underline{A})^K \cong C^p(G_\bullet, \underline{A})$
we have that $E_2^{p,0} \cong H^p( G_\bullet, \underline{A}).$ The result follows.
\end{proof}

\begin{proposition}[Lyndon-Hochschild spectral sequence] Under the same conditions of above, filtering the total complex $F'_p = \Tot(D^{*\geq p, *})$ by the 
degree $p$ in $D^{p,q}$, we obtain a spectral sequence whose second page is
$$E_2^{p,q} \cong H^p(K_\bullet, \underline{H^q(S_\bullet, \underline{A})})$$
which converges to $H^*(G_\bullet, \underline{A})$.
\end{proposition}
\begin{proof}
The associated graded group $\bigoplus_{p \geq 0} F'_p/F'_{p+1} $ is isomorphic
to the total complex but where we only take into account the differential $d_{K \rtimes G} $. The first page of the 
spectral sequence is
$$E_1^{q,p}\cong \bigoplus_{i+j=p}\Map\left (K^i, EB^j \left( H^q((K \rtimes G)_\bullet, \underline{A}) \right) \right)$$
and since $H^q((K \rtimes G)_\bullet, \underline{A}) \cong H^q(S_\bullet, \underline{A})$ we have that the first
page is 
$$E_1^{q,p}\cong \bigoplus_{i+j=p}\Map\left (K^i, EB^j \left( H^q(S_\bullet, \underline{A}) \right) \right).$$
The differential of the first page coincides with the differential of the complex $C^*(K_\bullet, \underline{H^q(S_\bullet, \underline{A})})$
and therefore the second page is
$$E_2^{q,p} \cong H^p(K_\bullet, \underline{H^q(S_\bullet, \underline{A})}).$$
The spectral sequence converges to $H^*(G_\bullet, \underline{A})$ by the previous proposition.
\end{proof}

\section{Segal-Mitchison gerbes} \label{section SM gerbes}

In this section we will introduce a homological model for gerbes based on Segal-Mitchison cohomology.
Let $A$ be an abelian group in $\Topab$. In what follows we will define the strict 2-category $A$-$\SMGrbs$ of Segal-Mitchison gerbes
with structure group $A$.

\subsection{Objects} \label{SM gerbes objects}
The objects are pairs $(M, \alpha)$ where $M$ is a paracompact and locally compact space and $\alpha$ belongs to $Z^2(M, \underline{A})$.

Note that the 2-cocycle is a map $\alpha: M \to EB^2A$ such that $\partial \circ \alpha $ is constant to the identity in $EB^3A$.
Since $B^3A \to EB^3A$ is an inclusion we have that the composition $M \stackrel{\alpha}{\to} EB^2A \to B^3A$ is constant to the 
identity. Therefore the image of $\alpha$ lies on the fiber of the identity and this fiber is isomorphic to $B^2A$. Therefore,
$Z^2(M, \underline{A}) \cong \Map(M, B^2A)$ and we may consider $\alpha$ as a map $\alpha: M \to B^2A$. 

Any Segal-Mitchison $A$-gerbe $(M, \alpha)$ defines a Lifting bundle $A$-gerbe $[\alpha^{*}EBA/EA]$ as defined in \S\ref{Appendix B}  Appendix B, where $\alpha^{*}EBA$ is the pullback of the canonical bundle $EBA\to B^2A$ along $\alpha:M \to B^2A$.

\subsection{Morphisms} \label{SM gebres morphisms}

Let $(M, \alpha)$ and $(N, \beta)$ be two Segal-Mitchison gerbes with structure group $A$. A morphism from 
$(M, \alpha)$ to $(N, \beta)$ is a pair $(F, c)$ where $F:M \to N$ is a continuous map and
$c \in C^1(M, \underline{A})$ such that 
$$\alpha - F^*\beta = \partial c.$$

This equation implies that the principal $BA$-bundles $\alpha^*EBA$ and $F^*\beta^* EBA$ are isomorphic, and that one can construct a bundle morphism $\alpha^{*}EBA \to \beta^{*}EBA$ covering $F:M \to N$. This bundle morphism determines a strict 1-morphism $[\alpha^{*}EBA/EA]\to [\beta^{*}EBA/EA]$ 
of lifting bundle $A$-gerbes in the following way. Let 
$$\alpha^{*}EBA := \{ (m, \lambda) \in M \times EBA | \alpha(m) = \partial(\lambda) \}$$
$$\beta^{*}EBA := \{ (n, \sigma) \in N \times EBA | \beta(n) = \partial(\sigma) \}$$
for $p : EBA \to B^2A$ the projection map and define the following map
\begin{align*}
[F,c] : \alpha^{*}EBA  &\to \beta^{*}EBA\\
(m,\lambda) & \mapsto (F(m), \lambda - c(m)).
\end{align*}
The structural equation $\alpha - F^*\beta = \partial c$ implies that $\beta(F(m))=\alpha(m) - \partial c(m)$
and therefore $[F,c]$ is well defined. Moreover, the map $[F,c]$ is $EA$-equivariant and therefore it defines
a strict morphism of lifting bundle $A$-gerbes
$$[F,c] : [\alpha^{*}EBA/EA]\to [\beta^{*}EBA/EA].$$

If we have morphisms
$$(M, \alpha)\stackrel{(F,c)}{\To}(N, \beta) \stackrel{(H,d)}{\To}(O, \gamma)$$
the composition is 
$$(H,d) \circ (F,c) = (H \circ F, F^*d + c).$$ This composition is associative, and it follows that the functor to lifting bundle $A$-gerbes strictly preserves this composition.

\subsection{2-Morphisms}  \label{SM gebres 2-morphisms}

If we have two morphisms between the same objects,  $(F_1,c_1), (F_2, c_2) : (M, \alpha) \to (N, \beta)$,
there are 2-morphisms between $(F_1,c_1)$ and $(F_2, c_2)$ whenever $F=F_1=F_2$. If this is the case, a
2-morphisms $$(F, c_1) \stackrel{e}{\Longrightarrow} (F,c_2)$$ consists of an element $e \in C^0(M, \underline{A})$
such that $$\partial e = c_2-c_1.$$

In the realm of lifting bundle $A$-gerbes the 2-morphism provides a natural transformation between the 
two morphisms of lifting bundle $A$-gerbes from $[(\alpha^*EBA)/EA]$ and $[(F^*\beta^*EBA)/EA]$. The natural
transformation is given by the map 
\begin{align*}
 \alpha^*EBA &\to F^*\beta^*EBA \times EA,\\
  (m,\lambda) &  \mapsto ((F(m), \lambda -c_1(m)),-\partial e(m))
\end{align*}
where $((F(m), \lambda -c_1(m)),-\partial e(m))$ is the morphism in $[(F^*\beta^*EBA)/EA]$ from 
$(F(m), \lambda -c_1(m))$ to $(F(m), \lambda -c_1(m)-\partial e(m))=(F(m), \lambda -c_2(m))$.

Note that the category $\Hom_{A\text{-}\SMGrbs}((M, \alpha),(N, \beta))$, whose objects are the morphisms of Segal-Mitchison gerbes
and whose 
morphisms are the 2-morphisms of Segal-Mitchison gerbes, defines a lifting bundle gerbe with structure group $H^0(M, \underline{A})$. In the notation
of \S \ref{Objects lifting bundle gerbes} it can be written as follows:
$$C^0(M, \underline{A}) \stackrel{\partial}{ \circlearrowright } \Hom_{A\text{-}\SMGrbs}^0((M, \alpha),(N, \beta)).$$

\subsection{Monoidal structure}
The 2-category $A$-$\SMGrbs$ can be endowed with a symmetrical monoidal structure. Take two
Segal-Mitchison gerbes $(M,\alpha)$ and $(N,\beta)$ and define their product
$$(M,\alpha) \times (N, \beta) := (M\times N, \pi_1^* \alpha + \pi_2^*\beta)$$
where where $\pi_1: M\times N \to M$ and $\pi_2: M\times N \to N$ are the projections.

This monoidal structure is compatible with the monoidal structure of the associated lifting bundle $A$-gerbes of
\S \ref{monoidal structure A-gerbes} of Appendix B. From the maps $\alpha : M \to B^2A$ and $\beta: N \to B^2A$
we  construct the map $M \times N \to B^2A$, $(m,n) \mapsto \alpha(m) + \beta(n)$ using the group
structure of $B^2A$. This map encodes the Segal-Mitchison $A$-gerbe of the 
product of lifting bundle $A$-gerbes $[(\alpha^*EBA)/EA]\otimes[(\beta^*EBA)/EA]$.

In \S \ref{monoidal structure A-gerbes} (Appendix B) we give a summary of the properties of the 
2-category of Lifting bundle $A$-gerbes and in Theorem \ref{theorem equivalence SM-gerbes LB-gerbes} we show its equivalence with the 2-category of Segal-Mithchison $A$-gerbes.



\section{Multiplicative Segal-Mitchison gerbes} \label{section multiplicative}
Multiplicative Segal-Mitchison gerbes will be nothing else but weak monoid objects in the monoidal 2-category $A$-$\SMGrbs$ of Segal-Mitchison gerbes.
They also form a 2-category which will be defined in the following.

\subsection{Objects}
A multiplicative Segal-Mitchison gerbe consists of an object $(M,\alpha_1)$, a  morphism
$$(M,\alpha_1) \times (M, \alpha_1) \stackrel{(m,\alpha_2)}{\To} (M, \alpha_1),$$
 a 2-morphism which controls the defect of the associativity, and a cocycle condition on the 2-morphism for the pentagon property.

Since at the level of spaces the maps are indeed associative, the space itself together with the product
$M \times M \stackrel{m}{\to} M$ is an associative monoid in spaces. We may consider the double complex
$C^{p,q}(M_\bullet, \underline{A})$ that the monoid $M$ induces since the maps that define the differential $\delta_M$ only
require the product of the monoid.
In terms of that double complex,  a multiplicative Segal-Mitchison gerbe over $M$ consists of:
\begin{itemize} \item A Segal-Mitchison gerbe $(M, \alpha_1)$ with $\partial \alpha_1=0$.
\item A morphism lifting the product in $M$
$$(M,\alpha_1) \times (M, \alpha_1) \stackrel{(m,\alpha_2)}{\To} (M, \alpha_1)$$
whose structural equation is
$$\pi_1^*\alpha_1 - m^*\alpha_1 + \pi_2^*\alpha_1 = \partial \alpha_2$$
with $\alpha_2 \in C^1(M^2 , \underline{A})$; this equation can be read as $$-\delta_M \alpha_1 +\partial \alpha_2=0$$
in the double complex  $C^{p,q}(M_\bullet, \underline{A})$ with $\alpha_1 \in C^{1,2}(M_\bullet, \underline{A})$ and 
$\alpha_2 \in C^{2,1}(M_\bullet, \underline{A})$.
\item  A 2-morphism $\alpha_3 \in C^0(M^3, \underline{A})$ controlling the defect of associativity
$$(m,\alpha_2) \circ (1_M \times m, \pi_{23}^*\alpha_2) \stackrel{\alpha_3}{\Longrightarrow} (m,\alpha_2) \circ (m \times 1_M , \pi_{12}^*\alpha_2)$$
where $\pi_{23}: M^3 \to M^2$ denotes the projection on the last two coordinates, $\pi_{12}:M^3 \to M^2$ the projection on the first two and $1_M$ denotes the identity map on $M$.
The structural equation becomes
$$\pi_{12}^*\alpha_2 - (1_M \times m)^*\alpha_2 + (m \times 1_M)^* \alpha_2 - \pi_{23}^* \alpha_2 = \partial \alpha_3$$
which can be read as
$$\delta_M \alpha_2 + \partial \alpha_3=0$$
in the double complex $C^{p,q}(M_\bullet, \underline{A})$ with $\alpha_3 \in C^{3,0}(M_\bullet, \underline{A})$.
\item And the pentagon condition for the 2-morphism $\alpha_3$ which can be simply written as
$$\delta_M \alpha_3=0$$
in the double complex $C^{p,q}(M_\bullet, \underline{A})$.
\end{itemize}

The element $\alpha : = \alpha_1 \oplus - \alpha_2 \oplus \alpha_3 $ becomes a 3-cocycle in $Z^3(M_\bullet, \underline{A})$
since 
$$(\delta_M +(-1)^p\partial) \alpha= (-\partial \alpha_1)  \oplus (\delta_M \alpha_1 -\partial \alpha_2)\oplus (-\delta_M \alpha_2 -\partial \alpha_3)  \oplus (\delta_M\alpha_3)=0,$$
and any 3 cocycle in $Z^3(M_\bullet, \underline{A})$ whose component in $C^{0,3}(M_\bullet, \underline{A})= \Map(pt, EB^3A)$
is zero defines a 2-monoid structure on Segal-Mitchison $A$-gerbes over the monoid $M$.

Note that the 0-th column $C^{0,*}(M_\bullet, \underline{A}) = \Map(pt, EB^qA)$ in the double complex is acyclic,
and moreover the differential $\delta_M$ is zero on all elements of this column. Therefore we remove the 0-th column from
 the complex, just leaving the coefficients $A$ at the $(0,0)$ coordinate, thus obtaining a better suited complex for understanding
 Segal-Mitchison multipliciative gerbes. Then define a new double complex $\overline{C}^{p,q}(M_\bullet, \underline{A})$ as
 the subcomplex of ${C}^{p,q}(M_\bullet, \underline{A})$ as follows:
 \begin{align*}
 \overline{C}^{p,q}(M_\bullet, \underline{A}) = \left\{
 \begin{array}{ccc}
 {C}^{p,q}(M_\bullet, \underline{A}) & \mbox{if} & q >0\\
 A & \mbox{if} & p=q=0 \\
 0 & \mbox{if} & q=0, p \neq 0.
  \end{array}
 \right.
 \end{align*}
Denote $\overline{C}^{*}(M_\bullet, \underline{A}) := \Tot ( \overline{C}^{*,*}(M_\bullet, \underline{A}))$
and by $\overline{Z}^{k}(M_\bullet, \underline{A}):= Z^k(\overline{C}^{*}(M_\bullet, \underline{A})$ its cocycles. Note that the inclusion
$$\overline{C}^{*}(M_\bullet, \underline{A}) \stackrel{\simeq}{\To} C^{*}(M_\bullet, \underline{A})$$
is a quasi-isomorphism and therefore $H^*(\overline{C}^{*}(M_\bullet, \underline{A}),d) \cong H^*(M_\bullet, \underline{A})$.
We now may claim the following result.
\begin{proposition}
Multiplicative Segal-Mitchison gerbes structures over the monoid  $M$ are in one-to-one correspondence with
elements in $\overline{Z}^3(M_\bullet, \underline{A})$.
\end{proposition}
\begin{proof}
We have seen by the construction of the 2-monoid that $\alpha : = \alpha_1 \oplus - \alpha_2 \oplus \alpha_3$
defines a 3-cocycle in $\overline{Z}^3(M_\bullet, \underline{A})$ since we removed the group corresponding to
 $C^{0,3}(M_\bullet, \underline{A})$.  Any 3-cocycle is also of this form and therefore it defines a multiplicative Segal-Mitchison
 $A$-gerbe over the monoid $M$.
\end{proof}

{\bf{Notation.}} We will denote multiplicative Segal-Mitchison $A$-gerbes by pairs $\langle M,\alpha \rangle$ where $M$ is a monoid 
and $\alpha  \in \overline{Z}^3(M_\bullet, \underline{A})$. All  information is encoded in this pair.

\subsection{Morphisms}

A morphism of multiplicative Segal-Mitchison gerbes 
$$\langle M,\alpha \rangle \stackrel{\langle F,\beta\rangle}{\To}\langle M',\alpha'\rangle $$
consists of a pair $\langle F, \beta\rangle$ with $F: M\to M'$ a morphism of monoids (a continuous map preserving the monoidal structure)
together with $\beta \in \overline{C}^{2}(M_\bullet, \underline{A})$ such that
$$\alpha - F^* \alpha' = d_M \beta =(\delta_M+(-1)^p \partial) \beta.$$ 

\subsection{2-Morphisms} Take two morphisms of multiplicative Segal-Mitchison gerbes
$$\langle F_1,\beta_1\rangle, \langle F_2,\beta_2 \rangle : \langle M,\alpha\rangle  \stackrel{\langle F,\beta\rangle }{\To}\langle M',\alpha' \rangle.$$
There exist 2-morphisms only when $F_1=F_2=F$, and in this case the 2-morphisms
$$\langle F,\beta_1\rangle \stackrel{\gamma}{\Longrightarrow} \langle F, \beta_2\rangle$$
consist of an element $\gamma \in \overline{C}^{1}(M_\bullet, \underline{A})$ such that
$$\beta_2-\beta_1 = d_M \gamma.$$

If we restrict the 2-category to a fixed monoid $M$ and we only consider morphisms which are the identity map on the monoid $M$,
we obtain the complex
\begin{align*}
\xymatrix{
\overline{C}^{1}(M_\bullet, \underline{A})   \ar[rr]^{d_M} && \overline{C}^{2}(M_\bullet, \underline{A}) \ar[rr]^{d_M} && \overline{Z}^{3}(M_\bullet, \underline{A}),
}
\end{align*}
whose cohomology groups are
$$H^1(M_\bullet,\underline{A}), \  H^2(M_\bullet,\underline{A}) \ \mbox{and} \ H^3(M_\bullet,\underline{A}).$$
Therefore we can classify multiplicative Segal-Mitchison $A$-gerbes over a fixed monoid $M$ by the following result.

\begin{proposition}
Isomorphism classes of multiplicative Segal-Mitchison gerbes with structure group $A$ over the monoid  $M$
are in 1-1 correspondence with the elements of the group $H^3(M_\bullet,\underline{A})$.
\end{proposition}

\section{Representations  of multiplicative Segal-Mitchison gerbes} \label{section representations}

In what follows we will represent multiplicative Segal-Mitchison gerbes on Segal-Mitchison gerbes. We will define the
 2-category $\Rep_{\langle M ,\alpha \rangle} $ of representations of a multiplicative Segal-Mitchison gerbe $\langle M ,\alpha \rangle$  by using the cohomological
nature that the Segal-Mitchison gerbes possess and we will provide explanations for the specific choices.

\subsection{Objects} A representation of the multiplicative Segal-Mitchison
gerbe $\langle M, \alpha \rangle$ consists of a pair $(N, \beta)$ where $N$ is a space endowed with an action of the monoid $M$
(in this case it is a right action but it may well be a left action). Denote the action by the map $\mu : N \times M \to N$. The second part is a 2-cochain $\beta \in C^2((N \rtimes M)_\bullet , \underline{A})$
such that 
$$d_{N \rtimes M}  \beta = \pi^* \alpha$$
where $\pi$ denotes the projection map $(N \rtimes M)_\bullet  \to M_\bullet$ which forgets the $N$-coordinate.
Let us call $(N,\beta)$ a $\langle M,\alpha \rangle$-representation.

Let us untangle the definition in order to see its structural pieces. Let $\beta = \beta_0 \oplus -\beta_1 \oplus \beta_2$
with $\beta_i \in C^{i,2-i}((N\rtimes M)_\bullet, \underline{A}) = \Map(N \times M^i, EB^{2-i}A)$ and recall
that $\pi^*\alpha : = \pi^*\alpha_1 \oplus - \pi^*\alpha_2 \oplus \pi^*\alpha_3$ with
 $\pi^*\alpha_j \in C^{j,3-j}((N\rtimes M)_\bullet, \underline{A})$.

\begin{itemize}
\item First thing to notice is that $\partial \beta_0=0$ and therefore $(N, \beta_0)$ is a Segal-Mitchison gerbe.
\item The second equation that one obtains is
$$\delta_{N \rtimes M} \beta_0 +\partial \beta_1 = \pi^*\alpha_1.$$
But since $\delta_{N \rtimes M} \beta_0 = \mu^* \beta_0 - \pi_1^* \beta_0$ we may rearrange the equation
as
$$\pi_1^*\beta_0 + \pi^*\alpha_1 - \mu^*\beta_0 = \partial \beta_1$$
which says that
$$(N,\beta_0) \times (M,\alpha_1 ) \stackrel{(\mu, \beta_1)}{\To} (N, \beta_0)$$
is a morphism of Segal-Mitchison gerbes.
\item The third equation is
$$-\delta_{N \rtimes M} \beta_1 +\partial \beta_2 = -\pi^*\alpha_2.$$
Unwinding it we obtain the equation for the defect of associativity
$$(\mu \times 1_M)^* \beta_1 + \pi_{12}^*\beta_1 - (1_N \times m)^* \beta_1 -\pi^* \alpha_2 = \partial \beta_2$$
associated to the diagram
$$
\xymatrix{
(N\times \beta_0) \times (M, \alpha_1) \times (M,\alpha_1) \ar[rrr]^-{(\mu \times 1_M,(\mu \times 1_M)^*\beta_1)} 
\ar[d]_{(1_N \times m, \pi^*\alpha_2)} &&&
 (N\times \beta_0) \times (M, \alpha_1) \ar[d]^{(\mu, \beta_1)}  \\
(N\times \beta_0) \times (M, \alpha_1) \ar[rrr]_-{(\mu,\beta_1)} \ar@{=>}[rrru]|*+{\beta_2} && &
(N\times \beta_0) .
}
$$
\item The last equation 
$$\delta_{N \rtimes M} \beta_2  = \pi^*\alpha_3$$
is simply the pentagon identity for the associator.
\end{itemize}

It should be now clear how to define the morphisms and 2-morphisms using the cohomological structure. We will 
not untangle the information anymore.

\subsection{Morphisms}
A morphism of multiplicative Segal-Mitchison gerbe representations
$$(N, \beta) \stackrel{(F, \gamma)}{\To} (N',\beta')$$
consists of a $M$-equivariant map $F: N \to N'$ and a cochain $\gamma \in C^1((N\rtimes M)_\bullet, \underline{A})$
such that
$$\beta - F^*\beta' = d_{N \rtimes M} \gamma.$$
The space of morphisms $\Hom_{\langle M, \alpha \rangle}((N, \beta),(N',\beta'))$ is endowed with the subspace
topology of $\Map_M(N,N') \times C^1((N\rtimes M)_\bullet, \underline{A})$ where the space of $M$-equivariant
 maps $\Map_M(N,N')$ is endowed with the compact open topology.
The composition of morphisms is continuous and is given by the formula
$$(F', \gamma') \circ (F,\gamma) = (F' \circ F, \gamma + F^*\gamma').$$

\subsection{2-Morphisms}
For two morphisms
$$(F_1,\gamma_1) , (F_2, \gamma_2) : (N, \beta) \To (N',\beta')$$
there are 2-morphisms from one to the other only if $F_1=F_2=F$. If this is the case
a 2-morphism $$(F, \gamma_1) \stackrel{\nu}{\Longrightarrow} (F, \gamma_2)$$
is a cochain $\nu \in C^0((N\rtimes M)_\bullet, \underline{A})$ such that
$$\gamma_2 - \gamma_1 = d_{N \rtimes M} \nu.$$

Fix a space $N$ where $M$ acts and a multiplicative Segal-Mitchison gerbe $\langle M,\alpha \rangle$. We see from the definition of a multiplicative Segal-Mitchison gerbe representation that a module over the space $N$ exists if and only if
$[\pi^*\alpha]=0$ in $H^3((N\rtimes M)_\bullet, \underline{A})$. If this is the case, the $\langle M,\alpha \rangle$-representation
structures over $N$ are (non-canonically) isomorphic with $Z^2((N\rtimes M)_\bullet, \underline{A})$, hence we have
the following result.

\begin{proposition}
Consider $\langle M,\alpha \rangle$ a multiplicative Segal-Mitchison gerbe with structure group $A$ and $N$ a space with a right action of $M$.
Then the space $N$ can be endowed with the structure of a representation
 of  $\langle M,\alpha \rangle$
if and only if $[\pi^*\alpha]=0$ in $H^3((N\rtimes M)_\bullet, \underline{A})$.
Whenever $[\pi^*\alpha]=0$, the isomorphism classes of $\langle M,\alpha \rangle$-representations with underlying space $N$ are in
(non-canonical) one-to-one correspondence with the group $H^2((N\rtimes M)_\bullet, \underline{A})$.
\end{proposition}

\begin{example}[Canonical representation] \label{Canonical representation}
Let us describe the canonical representation of a  multiplicative Segal-Mitchison gerbe $\langle M,\alpha \rangle$ with structure group $A$. 
Recall that $\alpha= \alpha_1 \oplus - \alpha_2 \oplus \alpha_3$ with 
$$\alpha_j \in C^{j,3-j}(M_\bullet, \underline{A}) = \Map(M^j, EB^{3-j}A).$$
Let $\overline{\alpha}_j: =\alpha_{j+1}$ and consider them as an elements in
$$\overline{\alpha}_j \in C^{j,2-j}((M\rtimes M)_\bullet, \underline{A}) = \Map(M \times M^j, EB^{2-j}A)$$
where the right action of $M$ over $M$ is given by right multiplication.
Then $\overline{\alpha}= \overline{\alpha}_1 \oplus  -\overline{\alpha}_2 \oplus \overline{\alpha}_3 \in C^2((M\rtimes M)_\bullet, \underline{A})$ and it satisfies the equation
$$d_{M \rtimes M} \overline{\alpha}= (\delta_{M \rtimes M} + (-1)^p \partial) \overline{\alpha} = \pi^* \alpha$$
where $\pi: (M \rtimes M)_\bullet \to M_\bullet$ forgets the first coordinate. 
Notice that
$$\pi^* \alpha - \delta_{M \rtimes M} \overline{\alpha} = \delta_{M} \alpha, \ \ \ \mbox{and} \ \ \ (-1)^p \partial \overline{\alpha}= (-1)^{p+1} \partial \alpha$$
and therefore we get the desired equation
\begin{align}
d_{M\rtimes M} \overline{\alpha}=(\delta_{M \rtimes M} + (-1)^p \partial )\overline{\alpha} = \pi^* \alpha - (\delta_M \alpha + (-1)^p \partial)\alpha = \pi^* \alpha. \label{eqn overline alpha}
\end{align}
Therefore, $(M, \overline{\alpha})$ is a representation of $\langle M , \alpha \rangle$ and we call it the canonical representation.
\end{example}

\subsection{Category of morphisms between representations} \label{category of morphisms of representations}
Note that the category 
$$\Hom_{\langle M , \alpha \rangle^{op}}((N, \beta), (N',\beta'))$$
 of morphisms 
of multiplicative Segal-Mitchison gerbe representations is a lifting bundle gerbe with structure group $H^0(N \rtimes M, \underline{A}) = \Map(N, A)^M$, the group of $M$-invariant maps from $N$ to $A$. We have
added the superscript $op$ to recall that the action of the multiplicative Segal-Mitchison gerbe acts from the right. Using the notation of \S \ref{Objects lifting bundle gerbes} the category
of morphisms of multiplicative Segal-Mitchison representations has the following form:
$$C^0(N \rtimes M, \underline{A}) \stackrel{d_{N \rtimes M}}{\circlearrowright} \Hom_{\langle M , \alpha \rangle^{op}}^0((N, \beta), (N',\beta')).$$
The underlying category is an action groupoid where the abelian group $C^0(N \rtimes M, \underline{A})$ acts on 
the morphisms $\Hom_{\langle M , \alpha \rangle^{op}}^0((N, \beta), (N',\beta'))$ by adding the image of the differential $d_{N \rtimes M}$.

\subsection{Endomorphisms of representations}

We will be interested in the category of endomorphisms of a right $\langle M,\alpha \rangle$-representation $(N,\beta)$. Denote this category by
$$\End_{\langle M,\alpha \rangle^{op}}(N,\beta):= \Hom_{\langle M,\alpha \rangle^{op}} \left((N,\beta), (N, \beta)\right)$$
and note that its space of objects fits in the sequence of monoids
$$Z^1((N \rtimes M)_\bullet, \underline{A}) \To \End^0_{\langle M,\alpha \rangle^{op}}(N,\beta) \To \End_{M^{op}}(N,[\beta]) $$
where the monoid  $\End_{M^{op}}(N,[\beta])$ denotes the space of right $M$-equivariant maps $F: N \to N$ such that the cohomology class
$[\beta- F^*\beta]$ is trivial. The space of morphisms of this category, $\End^1_{\langle M,\alpha \rangle^{op}}(N,\beta)$, is parameterized by 
$$ \End^0_{\langle M,\alpha\rangle^{op} }(N,\beta)) \times  C^0((N\rtimes M)_\bullet, \underline{A})  $$
whose product structure is
$$((F_1,\gamma_1), \nu_1) \cdot ((F_2,\gamma_2), \nu_2) = ((F_1 \circ F_2, F_2^*\gamma_1 +\gamma_2),F_2^*\nu_1 +\nu_2).$$

Note furthermore that the space of isomorphism classes of objects fits in the sequence of monoids
$$H^1((N \rtimes M)_\bullet, \underline{A}) \To \pi_0 \left( \End_{\langle M,\alpha \rangle^{op}}(N,\beta) \right)\To \End_{M^{op}}(N,[\beta])$$
and the group of endomorphisms of the identity map is the space of right $M$-invariant maps over $N$:
$$ H^0((N \rtimes M)_\bullet, \underline{A}) = \Map(N, A)^M.$$

The underlying four term sequence of the monoid in lifting bundle gerbes is
$$\Map(N, A)^M \to C^0((N\rtimes M)_\bullet, \underline{A}) \to 
\End^0_{\langle M,\alpha \rangle^{op}}(N,\beta) \To \pi_0 \left( \End_{\langle M,\alpha \rangle^{op}}(N,\beta) \right),$$
and whenever $ \End_M(N,[\beta])$ is a group, it becomes the four term sequence
of the crossed module
$$C^0((N\rtimes M)_\bullet, \underline{A}) \to 
\End^0_{\langle M,\alpha \rangle^{op}}(N,\beta)$$
associated to the 2-group. Note that $\End_{\langle M,\alpha \rangle^{op}}(N,\beta)$ is also a group object in 
lifting bundle gerbes as it is explained in  Appendix B \S \ref{Group lifting bundle gerbes}. In this case the structure group of the gerbe
is the group $\Map(N, A)^M$ of $M$-invariant maps.

\begin{example}[Endomorphisms of the canonical representation] \label{canonical representation}
Consider a multiplicative Segal-Mitchison gerbe $\langle G , \alpha \rangle$ where the monoid is a topological group $G$, and
$(G, \overline{\alpha})$ its canonical representation defined on Example \ref{group object lifting bundle gerbe}. Here
the group $G$ acts by multiplication on the right on $G$.

The morphisms of the category of endomorphisms fits in the short exact sequence of groups
\begin{align} \label{Extension of groups of canonical representation}
Z^1((G \rtimes G)_\bullet, \underline{A}) \To \End^0_{\langle G,\alpha \rangle^{op}}(G,\overline{\alpha}) \to G
\end{align}
since $\End_{G^{op}}(G,G) \stackrel{\cong}{\to} G, \ F \mapsto F(1)$ and $H^*((G \rtimes G)_\bullet, \underline{A})$ is trivial for $*>0$. The associated crossed module, or group object in Lifting bundle gerbes (see Lemma \ref{group object lifting bundle gerbe}), produces a four term exact sequence
$$\xymatrix{  A \ar[r] & C^0((G\rtimes G)_\bullet, \underline{A}) \ar[rr]^{d_{G \rtimes G}} &&
\End^0_{\langle G,\alpha \rangle^{op}}(G,\overline{\alpha}) \ar[r] & G}$$
where $A$ sits inside $ C^0((G\rtimes G)_\bullet, \underline{A}) = \Map(G,EA)$ as the $G$ invariant maps to $A$. 
Note that the crossed module, or group object in Lifting bundle $A$-gerbes
$$\xymatrix{ C^0((G\rtimes G)_\bullet, \underline{A}) \ar[rr]^{d_{G \rtimes G}} &&
\End^0_{\langle G,\alpha \rangle^{op}}(G,\overline{\alpha})}$$
encodes the information of the multiplicative
Segal-Mitchison gerbe $\langle G , \alpha \rangle$. 

Note that the complex $C^*((G \rtimes G)_\bullet, \underline{A})$ is a complex of $G$-$\Topab$ soft sheaves 
and moreover it is acyclic. 
 Therefore the short exact sequence
 $$
 \xymatrix{
 A \ar[r] & C^0((G\rtimes G)_\bullet, \underline{A}) \ar[r]^{d_{G \rtimes G}} &Z^1((G \rtimes G)_\bullet, \underline{A}) 
 }$$
 induces an isomorphism
 $$ H^2(G_\bullet, \underline{Z^1((G \rtimes G)_\bullet, \underline{A})}) \stackrel{\cong}{\To} 
 H^3(G_\bullet, \underline{A}). $$
 The inverse map of this isomorphism
 is given by the class in $H^2(G_\bullet, \underline{Z^1((G \rtimes G)_\bullet, \underline{A})})$ that
 encodes
 the group extension of \eqref{Extension of groups of canonical representation}.

Note that this 2-group of endomorphisms $\End_{\langle G,\alpha \rangle^{op}}(G,\overline{\alpha})$ is also group
object in Lifting bundle $A$-gerbes as it is explained in  \S \ref{Appendix B} of Appendix B,  Lemma \ref{group object lifting bundle gerbe}.
In this case the structure group of the gerbe is $A$ since the kernel of the differential $d_{G \rtimes G}$ consists of the $G$-invariant
$A$-valued maps over the transitive space $G$.

The 2-group of endomorphisms $\End_{\langle G,\alpha \rangle^{op}}(G,\overline{\alpha})$ encodes the same topological information
as the multiplicative Segal-Mitchison gerbe $\langle G, \alpha \rangle $. The first is a strict group object in Lifting bundle $A$-gerbes, the second
is a weak monoid object in Segal-Mitchison gerbes. The former could be understood as a strictification of the latter.

\end{example}

\section{Pontrjagin duality on multiplicative Segal-Mitchison $U(1)$-gerbes} \label{section pontrjagin}

In this section we will present a procedure by which we construct, from a specific choice of multiplicative
Segal-Mitchison $U(1)$-gerbe with a representation, another multiplicative Segal-Mitchison $U(1)$-gerbe. This second
multiplicative gerbe will be called fibrewise Pontrjagin dual because it generalizes Pontrjagin's duality
on locally compact topological abelian groups \cite{Morris,Pontrjagin}.

Recall that for a topological abelian group $S$ its Pontrjagin dual group is $\widehat{S}:=\Hom(S,U(1))$. 
Whenever $S$ is locally compact,  $\widehat{S}$ is also locally compact and the double Pontrjagin dual is canonically isomorphic to the original group $S \cong \widehat{\widehat{S}}$.
In what follows
we will use the notation and results of section \S\ref{section LHSS}.

Let us consider a central extension of  locally contractible and compactly generated groups
$$S \To G \To K$$
with $S$ locally compact. Let us take a cohomology class in $H^3(G_\bullet, \underline{U(1)})$ that vanishes once restricted 
to $H^3(S_\bullet, \underline{U(1)}) \cong E_2^{0,3} \supseteq E_\infty^{0,3}$ and once restricted to $E_\infty^{1,2}$ in the Lyndon-Hochschild spectral sequence of section \S \ref{section LHSS}. 
If we denote by $\Omega(G,S)$ those classes, the group is defined as the following subgroup of $H^3(G_\bullet, \underline{U(1)})$ :
$$\Omega(G,S) := \ker \left( \ker \left( H^3(G_\bullet, \underline{U(1)}) \to E_\infty^{0,3} \right) \to E_\infty^{1,2} \right).$$

This group of cohomology classes fits into the short exact sequence
$$0 \to E_\infty^{3,0} \To \Omega(G,S) \To  E_\infty^{2,1} \To 0$$
where the left hand side is the pullback of the classes from $K$
$$E_\infty^{3,0} =\Im (H^3(K_\bullet, \underline{U(1)}) \to H^3(G_\bullet, \underline{U(1)}) ),$$
and the right hand side is the cohomology of the middle term in the sequence
$$H^2(S_\bullet, \underline{U(1)}) \stackrel{d_2}{\To} H^2(K_\bullet, \underline{\widehat{S}}) 
\stackrel{d_2}{\To} H^4(K_\bullet, \underline{U(1)})$$
with $d_2$ the second differential of the spectral sequence. Here we have used the fact that
$\widehat{S} \cong H^1(S_\bullet, \underline{U(1)})$.

\vspace{0.3cm}

Take a cocycle $\alpha \in Z^3(G_\bullet, \underline{U(1)})$ whose cohomology class lies in $\Omega(G,S)$. Consider the restriction homomorphism
$$H^3(G_\bullet, \underline{U(1)}) \stackrel{\pi^*}{\to} H^3((K \rtimes G)_\bullet, \underline{U(1)}) \cong H^3(S_\bullet, \underline{U(1)})$$
and note that $[\alpha] \in \Omega(G,S)$ implies that $[\alpha|_S]=0$. Therefore  $[\pi^* \alpha]=0$ in $H^3((K \rtimes G)_\bullet , \underline{U(1)})$ and thus there is a representation $(K,\beta)$ with $\beta \in C^2(K \rtimes G , \underline{U(1)})$
such that $d_{K \rtimes G}  \beta = \pi^* \alpha$. We should think of 
$\beta$ as an element in $D^{0,2}$ in the double complex defined in the notation of \eqref{definition D**}
satisfying the equation 
\begin{align}
 d_{K \times G} \beta= \pi^* \alpha. \label{eqn beta}
\end{align}

Recall that the double complex is
\begin{align*} 
D^{p,q} = \bigoplus_{{i+j=p}, {l+m=q}}  C^{i,j,l,m}
\end{align*}
with differentials $d_K: D^{p-1,q} \to D^{p,q}$, $d_K:=\delta_K +(-1)^i \partial_1$, and $d_{K \rtimes G}: D^{p,q-1} \to D^{p,q}$,
$d_{K \rtimes G}:=\delta_{K \rtimes G} +(-1)^{l}\partial_2$, and the total complex
$\Tot(C^{*,*,*,*})=\Tot(D^{*,*})$ has for total differential $d= d_K +(-1)^{p}d_{K \rtimes G}.$

Since $[\alpha]$ belongs to $\Omega(G,S)$, the cohomology class $[\alpha]$ maps to zero in the cohomology group
$E_\infty^{1,2}$ and therefore there must exist $\gamma \in D^{1,1}$
such that $$d_K \beta - d_{K \rtimes G} \gamma=0.$$
Therefore $ d_K\gamma \in D^{2,1}$ satisfies the equations 
$$d_K (d_K \gamma) = 0 \ \ \mathrm{and} \ \  d_{K \rtimes G}(d_K\gamma)=0.$$
Hence we have that $$d_K \gamma \in Z^2(K_\bullet, \underline{Z^1((K \rtimes G)_\bullet, \underline{U(1)})}),$$
and the cocycles $\pi^*\alpha$ and $d_K \gamma$ are cohomologous in $\Tot^3(C^{*,*,*,*})$.

Recall that the inclusion $(pt \rtimes S) \to (K \rtimes G)$ of action groupoids induces an equivalence
of groupoids. Therefore the restriction map
$$Z^1((K \rtimes G)_\bullet, \underline{U(1)}) \to Z^1((pt \rtimes S)_\bullet, \underline{U(1)}) \cong \Hom(S, U(1))= \widehat{S}$$
induces a homomorphism
$$Z^2(K_\bullet, \underline{Z^1((K \rtimes G)_\bullet, \underline{U(1)})}) \to Z^2(K_\bullet, \underline{\widehat{S}}), \ \ d_1 \gamma \mapsto \widehat{F}.$$
This cocycle $\widehat{F}$ defines a central extension of the form $\widehat{S} \to \widehat{G} \to K$.

We will see in what follows that by choosing $[\alpha] \in \Omega(G,S)$ we make sure that there is a representation $(K,\beta)$ of the multiplicative Segal-Mitchison gerbe $\langle G, \alpha \rangle$, such that all the right $G$-equivariant homomorphisms $f_h \in\Map_{G^{op}}(K,K)$ with $f_h(k)=hk$ and $h \in K$
may be lifted to morphisms of representations from $(K,\beta)$ to $(K, f_h^*\beta)$.

Having set up the structure we claim the following result.

\begin{theorem} \label{Theorem crossed module}
Let $\langle G ,\alpha \rangle$ be a multiplicative Segal-Mitchison $U(1)$-gerbe such that $[\alpha] \in \Omega(G,S)$,
with $S \to G \to K$ a central extension by the locally compact abelian group $S$, and let $(K,\beta)$ be a right representation
of $\langle G ,\alpha \rangle$ as above. Then, 
the endomorphism category $\End_{\langle G ,\alpha \rangle^{op}}(K,\beta)$ is a 
crossed module, or group lifting bundle $A$-gerbe, whose four term sequence fits into the middle column of the diagram 
$$\xymatrix{
U(1)\ar[r] \ar[d] &U(1) \ar[d]& \\
C^0((K \rtimes G  )_\bullet, \underline{U(1)}) \ar[r] \ar[d]& C^0((K \rtimes G  )_\bullet, \underline{U(1)}) \ar[d] & \\
Z^1((K \rtimes G  )_\bullet, \underline{U(1)}) \ar[d]\ar[r] & \End^0_{\langle G ,\alpha \rangle^{op}}(K,\beta) \ar[r] \ar[d]& K\ar[d] \\
\widehat{S} \ar[r] & \widehat{G} \ar[r] & K. 
}$$  
\end{theorem}

\begin{proof}
Since the class $[\alpha]$ belongs to $\Omega(G,S)$, we have shown above that there must exist $\gamma \in D^{1,1}$
such that $d_K\beta - d_{K \rtimes G} \gamma=0.$ Therefore $d_K\gamma$ and $\pi^*\alpha$ are cohomologous
and $d_K \gamma \in Z^2(K_\bullet, \underline{Z^1((K \rtimes G)_\bullet, \underline{U(1)})})$ defines the cocycle
$\widehat{F} \in Z^2(K_\bullet, \underline{\widehat{S}})$ by the canonical map induced by the restriction
\begin{align*}r:Z^1((K \rtimes G)_\bullet,\underline{U(1)}) \to \Hom(S, U(1))=\widehat{S}.
\end{align*}

Since we have the isomorphism of topological groups
 \begin{align}
\Map_{G^{op}}(K,K) \cong K,   \ \ f_h \mapsto h \label{eqn map(K,K)}
\end{align}
where $f_h(k)=hk$ for $h,k \in K$,
we know that the existence of $\gamma$ implies that for all $h  \in K$
\begin{align} \label{eqn gamma}
[d_K(\beta)](h)= \beta - f_h^*\beta  = (d_{K \rtimes G}\gamma)(h).
\end{align}
Therefore the projection map
\begin{align} \label{eqn projection EndG(K,K)}
\End^0_{\langle G ,\alpha \rangle^{op}}(K,\beta) \to \Map_{G^{op}}(K,K) \cong K
\end{align}
is surjective.

Now,  $\gamma$ lives in $\Map(K, E C^1((K \rtimes G)_\bullet,\underline{U(1)})) $ and
for every $h \in K$ it chooses an element in $C^1((K \rtimes G)_\bullet,\underline{U(1)})$ such that $\beta - f_h^*\beta =(d_{K \rtimes G}\gamma)(h)$.
The class $d_K\gamma \in Z^2(K_\bullet, \underline{Z^1((K \rtimes G)_\bullet,\underline{U(1)})})$
measures the obstruction for the assignment $\gamma$ to be a section of the homomorphism
$\End^0_{\langle G ,\alpha \rangle}(K,\beta) \to K$. The cocycle $\widehat{F}$ defined
by $d_K\gamma$ through the restriction map
$$r_*:Z^2(K_\bullet, \underline{Z^1((K \rtimes G)_\bullet, \underline{U(1)})}) \to Z^2(K_\bullet, \underline{\widehat{S}}), \ \ d_K \gamma \mapsto r_*(d_K \gamma)=\widehat{F}$$
is equivalently defined as the cocycle in $Z^2(K_\bullet, \underline{H^1(S_\bullet,\underline{U(1) })})$ that
$d_K\gamma$ defines once one takes isomorphism classes of morphisms in
$\End_{\langle G ,\alpha \rangle}(K,\beta)$. Therefore we have the isomorphism of groups
$$\widehat{G} \cong \pi_0(\End_{\langle G ,\alpha \rangle^{op}}(K,\beta) )$$
and $\widehat{F}$ defines the extension class of the exact sequence
$$\widehat{S} \To \pi_0(\End_{\langle G ,\alpha \rangle^{op}}(K,\beta) ) \To K.$$
The rest of the diagram follows from the four term exact sequence
$$U(1) \to C^0((K \rtimes G  )_\bullet, \underline{U(1)}) \to Z^1((K \rtimes G  )_\bullet, \underline{U(1)}) \to \widehat{S}$$
since $H^1((K \rtimes G  )_\bullet, \underline{U(1)}) = \widehat{S}$.
\end{proof}

The crossed module, or group lifting bundle $A$-gerbe, and its four term sequence
$$U(1) \To C^0((K\rtimes G),\underline{U(1)}) \stackrel{d_{K \rtimes G}}{\To} \End^0_{\langle G ,\alpha \rangle^{op}}(K,\beta)  \to \widehat{G}$$
defines a cohomology class in $H^3(\widehat{G}_\bullet, \underline{U(1)})$ which is constructed as follows.

The central extension $\widehat{S} \to G \stackrel{\widehat{p}}{\to} K$ defines an element $\widehat{F}=(-\widehat{F}_1,\widehat{F}_2) \in Z^2(K_\bullet, \underline{\widehat{S}})$ whose components
are $\widehat{F}_1 : K \to B\widehat{S}$ and $ \widehat{F}_2:K\times K \to E\widehat{S}$ satisfying the equations
$$\delta_K\widehat{F}_1 (k_1,k_2) = \widehat{F}_1(k_2) - \widehat{F}_1(k_1k_2) + \widehat{F}_1(k_1)= \partial \widehat{F}_2(k_1,k_2), \ \ \delta_K\widehat{F}_2=0.$$
Let $\widehat{F}_1':\widehat{G} \to E\widehat{S}$ be the map of $\widehat{S}$-principal bundles making the diagram commutative
$$\xymatrix{
\widehat{G} \ar[d]_{\widehat{p}} \ar[r]^{\widehat{F}_1'} \ar[d] & E\widehat{S}\ar[d] \\
K \ar[r]^{\widehat{F}_1} & B\widehat{S}.}$$

Then the map $\widehat{F}'_1$ could be seen as living in $C^1(\widehat{G}, \underline{\widehat{S}})$ and it satisfies the equations
 $$\delta_{\widehat{G}}\widehat{F}'_1=\widehat{p}^*\widehat{F}_2, \ \ \ \partial \widehat{F}'_1= \widehat{p}^*\widehat{F}_1, $$
hence $(\delta_{\widehat{G}}+(-1)^l \partial) \widehat{F}'_1= d_{\widehat{G}} \widehat{F}'_1=\widehat{p}^*\widehat{F}$.

Let us carry out the same construction for the central extension $S \to G \stackrel{p}{\to} K$. Define
its extension class $F=(-F_1,F_2) \in Z^2(K_\bullet,\underline{S})$ and the map of $S$-principal bundles $F'_1 G \to ES$ satisfying
 $\delta_GF'_1= p^*F_2$ and $\partial F'_1=p^*F_1$.
 
 Define the cocycle $E =(E_0,E_1) \in Z^1((K \rtimes G)_\bullet, \underline{S})$ with $E_0 \in \Map(K, BS)$ and $E_0=F_1$ by the equations
 $$ E_0=F_1, \ \ E_1(k,g)=F_1'(g)-F_2(k,p(g)).$$
 A simple calculation shows that the following equations are satisfied
 \begin{align*}
 \delta_{K \rtimes G} E_0 (k,g) = F_1(kg)-F_1(k)&=F_1(p(g))-\partial F_2(k, p(g)) = \partial E_1 (k,g)\\
  \delta_{K \rtimes G} E_1 (k,g_1,g_2)&= 0.
  \end{align*}
  Any element $\rho \in \widehat{S}$ induces homomorphisms $ES \to EU(1)$ and $BS \to BU(1)$ and therefore
  we may define the homomorphism
  $$\psi : \widehat{S} \to Z^1((K \rtimes G)_\bullet, \underline{U(1)}), \ \ \rho \mapsto \psi(\rho) =\rho \circ (E_0,E_1),$$
  which is an inverse to the restriction homomorphism $Z^1((K \rtimes G)_\bullet, \underline{U(1)}) \to \widehat{S}$.
  Nevertheless, the map $\psi$ does not preserve the structure of $K$-module and this deficiency
  is quite important.
  
   Let us consider the homomorphism 
   $$C^1(\widehat{G}_\bullet,\underline{\widehat{S}}) \stackrel{\psi_*}{\To} C^1(\widehat{G}_\bullet,\underline{Z^1((K \rtimes G)_\bullet, \underline{U(1)})})$$
  induced by the map $\psi$ at the level of the coefficients and take $\psi_*\widehat{F}'_1$ the image of $\widehat{F}'_1$.
  The element $\psi_*\widehat{F}'_1 : \widehat{G} \to EZ^1((K \rtimes G)_\bullet, \underline{U(1)})$ should be understood as the lift of the map $\widehat{F}'_1:  \widehat{G} \to E\widehat{S} $ to 
  $Z^1((K \rtimes G)_\bullet, \underline{U(1)})$ in the left vertical map of the diagram
  $$\xymatrix{
  Z^1((K \rtimes G)_\bullet, \underline{U(1)}) \ar[r] \ar[d]& \End^0_{\langle G ,\alpha \rangle^{op}}(K,\beta)\ar[d] \ar[r] & K \ar@{=}[d] \\
  \widehat{S} \ar[r] & \widehat{G} \ar[r] & K. 
  }$$
  
  The failure of the map $\psi_*\widehat{F}'_1$ to induce a section is measured by the cocycle
  $$d_{\widehat{G}}  (\psi_*\widehat{F}'_1) \in Z^2(\widehat{G}_\bullet,\underline{Z^1((K \rtimes G)_\bullet, \underline{U(1)})}),$$
and its restriction to  $Z^2(\widehat{G}_\bullet,\underline{\widehat{S}})$ is 
 $$r_* \left( d_{\widehat{G}} (\psi_*\widehat{F}'_1) \right)= \widehat{p}^*\widehat{F}.$$
 
The element $$[d_{\widehat{G}} (\psi_*\widehat{F}'_1) - \widehat{p}^*(d_K\gamma)] \in Z^2(\widehat{G}_\bullet,\underline{Z^1((K \rtimes G)_\bullet, \underline{U(1)})})$$ restricts to zero in 
$Z^2(\widehat{G}_\bullet,\underline{\widehat{S}})$ since 
 $r_*(\widehat{p}^*(d_K\gamma))= \widehat{p}^*(r_*(d_K\gamma)=\widehat{p}^*\widehat{F}$. Therefore 
 $$[d_{\widehat{G}} (\psi_*\widehat{F}'_1) - \widehat{p}^*(d_K\gamma)] \in Z^2(\widehat{G}_\bullet,\underline{B^1((K \rtimes G)_\bullet, \underline{U(1)})})$$ where $B^1((K \rtimes G)_\bullet, \underline{U(1)})$ are the coboundaries fitting in the short exact sequence
$$ B^1((K \rtimes G)_\bullet, \underline{U(1)}) \hookrightarrow Z^1((K \rtimes G)_\bullet, \underline{U(1)}) \stackrel{r}{\to} \widehat{S}.$$
 Now, the coboundaries also fit in the short exact sequence of abelian groups in $\widehat{G}$-$\Topab$
 $$U(1) \to C^0((K \rtimes G)_\bullet, \underline{U(1)}) \to B^1((K \rtimes G)_\bullet, \underline{U(1)})$$ thus inducing a map in cohomology
 \begin{align*}
H^2(\widehat{G}_\bullet,\underline{B^1((K \rtimes G)_\bullet, \underline{U(1)})}) & \to H^3(\widehat{G}_\bullet,\underline{U(1)})\\
[d_{\widehat{G}} (\psi_*\widehat{F}'_1) - \widehat{p}^*(d_K\gamma)]& \mapsto [\widehat{\alpha}_{(K,\beta)}].
\end{align*}

The explicit description of the cocycle $ \widehat{\alpha}_{(K,\beta)} \in Z^3(\widehat{G}_\bullet,
  \underline{U(1)})$ depends on several choices, but its cohomology class only depends on the multiplicative Segal-Mitchison gerbe $\langle G, \alpha \rangle$ 
  and the representation $(K, \beta)$. The cohomology class $ [\widehat{\alpha}_{(K,\beta)}]$ characterizes the group lifting bundle $A$-gerbe $\End_{\langle G, \alpha \rangle^{op}}(K,\beta)$.
  
Bundling up we obtain:
  
  \begin{theorem} \label{Theorem cohomology class}
  The crossed module, or group lifting bundle $A$-gerbe,
  $$C^0((K \rtimes G)_\bullet, \underline{U(1)}) \To \End^0_{\langle G ,\alpha \rangle^{op}}(K,\beta)$$
  induced by the category $\End_{\langle G ,\alpha \rangle^{op}}(K,\beta)$ of endomorphisms of the right $\langle G, \alpha \rangle$-representation $(K,\beta)$
  has $\widehat{G}$ for cokernel,  $U(1)$ as kernel and it is classified by the cohomology class
   $$[\widehat{\alpha}_{(K,\beta)}]  \in  H^3(\widehat{G}_\bullet,
  \underline{U(1)})$$
  defined above.
  \end{theorem}

  Both groups $G$ and $\widehat{G}$ project to $K$ and the fibers are
  $S$ and $\widehat{S}= \Hom(S, U(1))$ respectively. Since $S$ is a locally compact abelian group, its Pontrjagin dual group $\widehat{S}$ is also a locally compact abelian group \cite[First Fundamental Theorem]{Pontrjagin}. Therefore we propose the following definition.
  
\begin{definition} \label{definition pontrjagin dual}
  The multiplicative Segal-Mitchison gerbe $ \langle \widehat{G}, \widehat{\alpha}_{(K,\beta)} \rangle$ is a fibrewise Pontrjagin dual of the 
multiplicative Segal-Mitchison gerbe $\langle G, \alpha \rangle$.

\end{definition}

 Let us now see some specific cases.

  \subsection{Classical Pontrjagin duality} Take $S$ a locally compact abelian group and $\alpha=0$.
   By the Second Fundamental Theorem of Pontrjagin \cite[p. 377]{Pontrjagin}, cf. \cite[Thm. 2]{Morris}, the group $S$ is  locally compact and
  locally contractible and therefore it satisfies the conditions of Theorem 
  \ref{Theorem crossed module}. Consider
  the representation of $\langle S,0 \rangle$ to be $(*,0)$ where $*$ is a point.
  
  The 2-group $\End_{\langle S,0 \rangle}(*,0)$ is isomorphic to $U(1)[1] \times \widehat{S}$ since its four
term sequence is 
$$ U(1) \To  C^0((* \rtimes S)_\bullet, \underline{U(1)}) \To Z^1((* \rtimes S)_\bullet, \underline{U(1)}) \To \widehat{S}$$
  and its cohomology class in $H^3(\widehat{S}_\bullet, \underline{U(1)})$ is trivial because the last homomorphism splits through the 
inclusion of groups
$\widehat{S} = \Hom(S, U(1)) \subset  Z^1((* \rtimes S)_\bullet, \underline{U(1)})$.
  
  We see that the fibrewise dual of the multiplicative Segal-Mitchison $U(1)$-gerbe $\langle S, 0 \rangle$ is the multiplicative Segal-Mitchison $U(1)$-gerbe $\langle \widehat{S}, 0 \rangle$
  and vice versa, thus recovering Pontrjagin's duality in the context of multiplicative Segal-Mitchison gerbes.
  \subsection{Multiplicative Segal-Mitchison gerbe with trivial cocycle}
  Consider a non-trivial central extension $S \to G \to K$ and $\alpha =0$. Then
  we may take the representation to be $(K,0)$ and therefore we get that
  $$\End^0_{\langle G ,0 \rangle^{op}}(K,0) \cong K \ltimes Z^1((K \rtimes G)_\bullet, \underline{U(1)}).$$
  Hence the fibrewise Pontrjagin dual multiplicative Segal-Mitchison gerbe has for underlying group $\widehat{G}=K \times \widehat{S}$. Now consider $\widehat{F}'_1: K \times \widehat{S} \to E\widehat{S} $
  to be defined by the inclusion $\widehat{S} \to E\widehat{S}$.  Following the construction we get
  $(\psi_*\widehat{F}'_1) (k, \rho) = \rho \circ (E_0,E_1)$
  and 
  $$[d_{\widehat{G}} (\psi_*\widehat{F}'_1) ]((k_1, \rho_1),(k_2, \rho_2))= \rho_1(k_2^*(E_0,E_1)-(E_0,E_1)).$$
  The cocycle thus defined  $\widehat{\alpha}_{(K,0)}=(\widehat{\alpha}^3_{(K,0)},\widehat{\alpha}^2_{(K,0)})$ with $\widehat{\alpha}^3_{(K,0)}: \widehat{G}^3 \to EU(1)$ and $\widehat{\alpha}^2_{(K,0)}: \widehat{G}^2 \to BU(1)$ is defined in
  coordinates as follows:
  \begin{align} \label{alpha hat}
  \widehat{\alpha}^3_{(K,0)} ((k_1, \rho_1),(k_2, \rho_2), (k_3, \rho_3)) =\rho_1(F_2(k_2,k_3)) \in EU(1) \\ \nonumber
  \widehat{\alpha}^2_{(K,0)} ((k_1, \rho_1),(k_2, \rho_2)) = \rho_1(F_1(k_2)) \in BU(1).
  \end{align}
  Therefore the cohomology class $[\widehat{\alpha}_{(K,0)}] \in H^3(K \times \widehat{S}_\bullet, \underline{U(1)})$
  is simply obtained by the composition of $[F] \in H^2(K_\bullet, \underline{S})$ with $\widehat{S}=\Hom(S,U(1))$, and the Multiplicative Segal-Mitchison
  gerbe $\langle K \times \widehat{S},\widehat{\alpha}_{(K,0)}\rangle$ is a fibrewise
  Pontrjagin dual of $\langle G,0\rangle$.
 \subsection{Multiplicative Segal-Mitchison gerbe over a trivial extension}
  Let us start now with the trivial extension $\widehat{G} = K \times \widehat{S}$ and the cocycle 
  $\widehat{\alpha}_{(K,0)}$ defined by the equations \eqref{alpha hat}. In this case the dual group is 
  $$\pi_0(\End^0_{\langle K \times \widehat{S} ,\widehat{\alpha}_{(K,0)} \rangle}(K,0)) \cong G$$ with 
  cocycle $F \in Z^2(K_\bullet, \underline{S})$, and the cohomology class of the crossed module
  is trivial because $$d_{G}\psi_*F'_1= p^*\psi_* F$$
  and $\gamma$ may be chosen so that $d_G\gamma = \psi_* F$.
Therefore, we conclude:
\begin{proposition} \label{Cor dual gerbes} The multiplicative Segal-Mitchison $U(1)$-gerbes
$$\langle G, 0 \rangle  \ \ \mbox{and} \ \ \langle K \times \widehat{S}, \widehat{\alpha}_{(K,0)} \rangle$$
defined above are mutually fibrewise Pontrjagin dual.
\end{proposition}

\subsection{Groups whose connected components of the identity are contractible}
In the case that $G$, $\widehat{G}$, $S$, $\widehat{S}$ and $K$ are groups
whose connected components of the identity are contractible,
we can provide an explicit description of the class $  [\widehat{\alpha}_{(K,\beta)}]$ using the results of \cite[\S 3]{Uribe}.

For these specific choice of groups we know by Lemma \ref{lemma A contractible} that
$$H^*(G_\bullet, \underline{U(1)}) \cong H^*_{cont}(G,U(1))$$
where $H^*_{cont}(G,U(1))$ denotes the cohomology of the complex
of continuous cochains of the group $G$ with values in $U(1)$.

Following the notation, the formulas and the calculations of 
\cite[\S 3]{Uribe} we may take $G \cong S \rtimes_F K$ where $F \in Z^2_c(K,S)$ and the product
structure on  $S \rtimes_F K$ is given by the formula
$$(a_1,k_1)\cdot (a_2,k_2)= (a_1a_2F(k_1,k_2), k_1k_2).$$

From \cite[Thm. 3.6]{Uribe} we know that the 3-cocycle $\alpha \in Z^3_c(G,U(1))$ can be defined as
$$\alpha((a_1,k_1),(a_2,k_2), (a_3,k_3))= \widehat{F}(k_1,k_2)(a_3) \epsilon(k_1,k_2,k_3)$$
where $\widehat{F} \in Z^2(K, \widehat{S})$ and 
$$\delta_K \epsilon(k_1,k_2,k_3,k_4)= \widehat{F}(k_1,k_2) (F(k_3,k_4))$$
with $\epsilon \in C^3_c(K,U(1))$.
 
From this definition of $\alpha$ we know that $\widehat{G} \cong K \ltimes_{\widehat{F}} \widehat{S}$ with product structure
$$(k_1, \rho_1)\cdot (k_2,\rho_2)= (k_1k_2,\rho_1\rho_2\widehat{F}(k_1,k_2)).$$
From \cite[Thm. 3.6]{Uribe} we know that $\beta \in C^2_c(K \rtimes G, U(1))$ 
can be written as follows  
$$\beta(k_1, (a_2, k_2),(a_3, k_3)) = \left(\widehat{F}(k_1,k_2)(a_3) \epsilon(k_1,k_2,k_3) \right)^{-1}$$
and $\gamma \in C^1_c(K, C^1_c(K \rtimes G, U(1))$ by the equation
$$\gamma(k_1, (x_1,(a_2,x_2)))) = \widehat{F}(k_1,x_1)(a_2) \epsilon(k_1,x_1,x_2);$$
therefore $\widehat{p}^*d_1\gamma$ satisfies the following  equation
\begin{align*}
(\widehat{p}^*d_1 \gamma) ( (k_1, \rho_1), (k_2,\rho_2)) (x_1,(a_2,x_2))=&\\ \widehat{F}(k_1,k_2) & (a_2F(x_1x_2)) \epsilon(k_1,k_2,x_1x_2) 
\epsilon(k_1,k_2,x_1)^{-1}.\end{align*}

The cochain $\widehat{F}'_1$ can be defined by the
map $\widehat{F}'_1(k, \rho)= \rho^{-1}$
and therefore $\psi_* \widehat{F}'_1$ satisfies the following equation
$$\psi_* \widehat{F}'_1 (k,\rho)(x_1,(a_2,x_2))= \rho(a_2F(x_1,x_2))^{-1}.$$

Calculating the coboundary we obtain
\begin{align*}
\delta_{\widehat{G}} \psi_* \widehat{F}'_1 ((k_1, \rho_1), (k_2,\rho_2))(x_1,(a_2,x_2))= &\\ \widehat{F}(k_1,k_2)(a_2F(x_1,x_2)) &\rho_1(F(k_2,x_1)F(k_2,x_1x_2)^{-1})
\end{align*}
and therefore 
\begin{align*}
[\delta_{\widehat{G}} (\psi_*\widehat{F}'_1) - \widehat{p}^*(d_1\gamma)]&
((k_1, \rho_1), (k_2,\rho_2))(x_1,(a_2,x_2)) \\
 =& \epsilon(k_1,k_2,x_1x_2) ^{-1}
\epsilon(k_1,k_2,x_1) \rho_1(F(k_2,x_1x_2)^{-1}F(k_2,x_1)) \\
=& \delta_{K \rtimes G} \eta (((k_1, \rho_1), (k_2,\rho_2))(x_1,(a_2,x_2))
\end{align*}
where 
$\eta((k_1,\rho_1),(k_2,\rho_2),x))= \epsilon(k_1,k_2,x)^{-1} \rho_1(F(k_2,x))^{-1}$
with $\eta \in C^2_c(\widehat{G}, C^0_c(K \rtimes G, U(1)))$. The coboundary of $\eta$
satisfies the equation
$$\delta_{\widehat{G}}\eta((k_1,\rho_1),(k_2,\rho_2),(k_3, \rho_3),x) = \epsilon(k_1,k_2,k_3) \rho_1(F(k_2,k_3))$$
and it is independent of the variable $x$. We obtain that the dual class $\widehat{\alpha}_{(K, \beta)}$ can be written explicitly by the following formula
$$\widehat{\alpha}_{(K, \beta)}((k_1,\rho_1),(k_2,\rho_2),(k_3, \rho_3))= 
\epsilon(k_1,k_2,k_3) \rho_1(F(k_2,k_3)).$$

From the previous calculations we can conclude the following theorem
which generalizes \cite[Thm. 3.7]{Uribe}.
 
  \begin{theorem} \label{theorem identity contractible}
  Let $K$ be a compactly generated topological group
  and $S$ a locally compact and compactly generated topological abelian group, such that
  the connected components of the identity on $K$, $S$ and $\widehat{S}$ are contractible groups.
  Consider the groups $G = S \rtimes_F K$ and $\widehat{G} = K \ltimes_{\widehat{F}} \widehat{S}$ with $\widehat{S}$ the Pontrjagin dual of $S$, $F \in Z^2_c(K,S)$,
  $\widehat{F} \in Z^2_c(K,\widehat{S})$ such that $\delta_K \epsilon (k_1,k_2,k_3,k_4)= \widehat{F}(k_1,k_2)(F(k_3,k_4))$ for $\epsilon \in C^3_c(K,U(1))$.
  Define the 3-cocycles $\alpha \in Z^3_c(G,U(1))$ and $\widehat{\alpha} \in Z^3_c(\widehat{G},U(1))$ by the equations
  \begin{align*}
  \alpha((a_1,k_1),(a_2,k_2), (a_3,k_3))= & \widehat{F}(k_1,k_2)(a_3) \epsilon(k_1,k_2,k_3)\\
  \widehat{\alpha}((k_1,\rho_1),(k_2,\rho_2),(k_3, \rho_3))= &
\epsilon(k_1,k_2,k_3) \rho_1(F(k_2,k_3)).
  \end{align*}
Then the multiplicative Segal-Mitchison gerbes $\langle G, \alpha \rangle$
and $\langle \widehat{G}, \widehat{\alpha} \rangle$ are fibrewise Pontrjagin dual.
  \end{theorem}
 
\begin{example}
Consider $K= \real^2$, $S= \real$ and $G = K \times S = \real^3$. The Pontrjagin dual group $\widehat{S}= \Hom(S,U(1))$ is isomorphic to the reals and we may consider $\alpha \in Z^3_c(\real^3, U(1))$ by the formula $$\alpha((r_1,s_1,t_1),(r_2,s_2,t_2),(r_3,s_3,t_3))=e^{2\pi i r_1s_2t_3}.$$
Applying Proposition \ref{Cor dual gerbes} and Theorem \ref{theorem identity contractible} we see that the fibrewise Pontrjagin dual multiplicative
Segal-Mitchison gerbe is $\langle \mathcal{H}, 0 \rangle $ where $ \mathcal{H}$ is the Heisenberg group of upper triangular $3\times 3$ matrices
with ones on the diagonal and trivial 3-cocycle. In this case 
 $\widehat{F}((r_1,s_1),(r_2,s_2))=r_1s_2$ with $\widehat{F} \in Z^2_c(\real^2, \real)$ and $F=0$. 
We conclude that 
$$\langle \real^3, \alpha \rangle  \ \  \mbox{and} \ \ \langle \mathcal{H},0 \rangle$$ are fibrewise Pontrjagin dual multiplicative Segal-Mitchison gerbes.
\end{example}

\subsection{Finite abelian fiber}

Consider the case on which $S$ is a finite abelian group. In what follows we will provide an explicit construction of the crossed modules associated to the fibrewise Pontrjagin
dual multiplicative Segal-Mitchison $U(1)$-gerbes of Proposition \ref{Cor dual gerbes}.

Let $\widehat{S}:=\Hom(S,U(1))$ be the Pontrjagin dual finite group
 and take $b:S \times S \to U(1)$ a bilinear map inducing an isomorphism of groups $b^\sharp : S \stackrel{\cong}{\to} \widehat{S}$,
 $s\mapsto b(s,-)$.

Consider a topological group $G$ on which $S$ is a central subgroup fitting into the non-trivial group extension $S \to G \to K$.

We are going to define two non-equivalent crossed modules with this information
which are fibrewise Pontrjagin dual. They will have the same structural groups
$N$ and $E$, same action of $E$ on $N$ and same $\pi_1$ but different structural homomorphisms. 
Let us define the two 2-groups first.

Consider the groups $$\GG_0:= S \times G \ \ \mbox{and} \ \ \GG_1:= (S \times G \times S) \ltimes U(1)$$ where the product in $\GG_1$ is given by the formula
$$(s_1,g_1,\overline{s}_1,\lambda_1)\cdot (s_2,g_2,\overline{s}_2,\lambda_2) = (s_1s_2,g_1g_2,\overline{s}_1\overline{s}_2,\lambda_1+  \lambda_2 + b(\overline{s}_1, s_2)).$$
In both cases the underlying groupoid comes from the action of $S \times U(1)$ on $S\times G$. In one case
we let $S \times U(1)$ act on $S \times G$ by multiplication on the first coordinate, and in the other case
by multiplication on the second coordinate.
In both cases the source map is the same $\source(s,g, \overline{s},\lambda)=(s,g)$ and the identity map also
$\ident(s,g)=(s,g,0,0)$. Therefore the induced right action of $S\times G$ on $S \times U(1)$
is given by the following formula
$$(\overline{s},\lambda)^{(s,g)}=(\overline{s},\lambda + b(\overline{s},s)).$$
The target maps are defined as follows
$$\target_1(s,g, \overline{s},\lambda)= (s\overline{s},g), \ \ \ \ \target_2(s,g, \overline{s},\lambda)= (s, g\overline{s})$$
and therefore they define crossed modules $$S \times U(1) \stackrel{\target_1}{\To} S \times G  \ \ \mbox{and} \ \ S \times U(1) \stackrel{\target_2}{\To} S \times G$$ such that in both cases $\Ker(\target_1) \cong U(1) \cong \Ker(\target_2)$.

Nevertheless the crossed modules are non-equivalent since $\Coker(\target_1)= G$ and $\Coker(\target_2)=S \times K $.
Moreover, the cohomology class $[\alpha] \in H^3((S\times K)_\bullet, \underline{U(1)})$ that induces the second crossed module is not trivial.

These two crossed modules defined above are the fibrewise Pontrjagin dual multiplicative Segal-Mitchison gerbes of Proposition \ref{Cor dual gerbes}.

\begin{example} Consider $G=SU(2)$ and $S= \integer/2 \subset G$ its center. Then $K=SO(3)$ and we may take the bilinear
form $b : \integer/2 \times \integer/2 \to U(1)$, $(x,y) \mapsto \frac{xy}{2}$ with $x,y \in \{0,1\}$ and $U(1)=\real/\integer$.
The cohomology class of the crossed module that induces the four term sequence
$$U(1) \to \integer/2 \times U(1) \To \integer/2 \times SU(2) \To \integer/2 \times SO(3)$$
is the only non-trivial class that comes from the 
map
$$H^1(H^2(SO(3)_\bullet, \underline{\integer/2}),U(1)) \to H^3((\integer/2 \times SO(3))_\bullet,
\underline{U(1)}).$$

The two crossed modules previously defined induce multiplicative Segal-Mitchison $U(1)$-gerbes. On the one hand we have the group $SU(2)$ with trivial cohomology class and in the other hand we have the group $\integer/2 \times SO(3)$ with the cohomology class in $H^3((\integer/2 \times SO(3))_\bullet,
\underline{U(1)})$ defined by composition of the Stiefel-Whitney class $w_2 \in H^2(BSO(3), \integer/2)$ with the non-trivial homomorphism in $H^1((\integer/2)_\bullet,
\underline{U(1)})$.
\end{example}

\subsection{Finite groups}
Multiplicative Segal-Mitchison gerbes over finite groups and pointed fusion categories are very related entities.  A multiplicative Segal-Mitchison gerbe $\langle G, \alpha \rangle $ with $G$ a finite group and $\alpha \in Z^3(G,U(1))$ encodes the information of a pointed fusion category. A pointed fusion category is a rigid tensor category with finitely many isomorphism classes of simple objects which moreover are invertible \cite{ENO}.
The relevant pointed fusion categories are $\Vect(G,\alpha)$ of $G$-graded complex vector spaces with associativity constraint $\alpha$. 
Module categories over the pointed fusion categories $\Vect(G,\alpha)$ play the role of representations of the multiplicative Segal-Mitchison gerbe 
$\langle G, \alpha \rangle $ and vice versa.

The realm of fusion categories and their module categories has been extensively studied \cite{ENO, Ost, Ost-2, MugerI, MugerII, Naidu} and there are many properties and constructions
which do not necessarily generalize to the continuous case. Nevertheless, in the finite case we can compare our Theorem \ref{theorem identity contractible} with Theorem 3.9 in \cite{Uribe} and we can conclude  that Morita equivalent
pointed fusion categories (namely such that their module categories are equivalent) induce fibrewise Pontrjagin dual multiplicative Segal-Mitchison gerbes and that pointed fusion categories of fibrewise Pontrjagin dual multiplicative Segal-Mitchison gerbes are Morita equivalent. Therefore we can conclude that fibrewise Pontrjagin dual multiplicative Segal-Mitchison 
gerbes over finite groups have equivalent categories of representations. The generalization to multiplicative Segal-Mitchison gerbes over topological groups will be approached in the next section.

Concrete examples of fibrewise Pontrjagin dual multiplicative Segal-Mitchison gerbes can thus be obtained by looking at the examples that have been constructed
of Morita equivalent pointed fusion categories. A complete classification of all fibrewise Pontrjagin dual Segal-Mitchison gerbes
over groups of dimension 8  appears in \cite{MunozUribe}, over groups of order $p^3$ for $p$ odd prime  in \cite{MayaUribe} and over groups of order less than 32  in \cite{MignardSchauenburg}. Several other examples appear in \cite{Naidu, MugerI, Ost, Ost-2} and the references therein.

\section{Morita equivalence for fibrewise Pontrjagin dual multiplicative Segal-Mitchison gerbes} \label{section Morita}
In this section we will focus on the relation between the categories of representations of fibrewise Pontrjagin dual multiplicative Segal-Mitchison gerbes.
For this let us take a multiplicative Segal-Mitchison $U(1)$-gerbe $\langle G, \alpha \rangle$ and a its representation $(K,\beta)$
as it appears in Theorem \ref{Theorem crossed module}. In Theorem \ref{Theorem cohomology class} we have shown that
the 2-group of endomorphisms $\End_{\langle G, \alpha \rangle^{op}}(K,\beta)$, which is also a group object in lifting bundle $A$-gerbes, has $\widehat{G}$ for its underlying group and its classifying
class is $[\widehat{\alpha}_{(K,\beta)}] \in H^3(\widehat{G}_\bullet,\underline{U(1)})$. The multiplicative
Segal-Mitchison $U(1)$-gerbe $\langle \widehat{G}, \widehat{\alpha}_{(K,\beta)} \rangle$ was called a fibrewise Pontrjagin dual of $\langle G, \alpha \rangle$
in Definition \ref{definition pontrjagin dual}.

In order to define the functors relating the categories of representations of the Segal-Mitchison $U(1)$-gerbes we need to recall  first the properties
of the groupoids of morphisms of two representations from \S \ref{category of morphisms of representations}. 
Let us consider two right representations $(V, \chi)$ and $(W,\nu)$ of the Segal-Mitchison $U(1)$-gerbe $\langle G, \alpha \rangle$. The groupoid
of morphisms
$$\Hom_{\langle G, \alpha \rangle^{op}}((V, \chi),(W,\nu)) = [\Hom^0_{\langle G, \alpha \rangle^{op}}((V, \chi),(W,\nu)) / C^0((V \rtimes G  )_\bullet, \underline{U(1)}) ]$$
is an action groupoid where the group $C^0((V \rtimes G  )_\bullet, \underline{U(1)})$ acts on the space of objects $\Hom^0_{\langle G, \alpha \rangle^{op}}((V, \chi),(W,\nu))$ by the image of the differential $d_{V \rtimes G}$.
The kernel of the differential is $H^0((V \rtimes G  )_\bullet, \underline{U(1)}) \cong \Map(V,U(1))^G$, the space of $G$-invariant maps from $V$ to $U(1)$, and therefore the groupoid of morphisms becomes a Lifting bundle gerbe with structure group $\Map(V,U(1))^G$ as it described in  \S \ref{Appendix B} Appendix B.

The endomorphism 2-groups $\End_{\langle G, \alpha \rangle^{op}}(G, \overline{\alpha})$ and its Pontrjagin dual $\End_{\langle G, \alpha \rangle^{op}}(K, \beta)$
defined in Example \ref{canonical representation} and Theorem \ref{Theorem crossed module} respectively, are both
also group objects in Lifting bundle $U(1)$-gerbes. As such, they act on Lifting bundle gerbes and we will denote by
$$\Rep_{\End_{\langle G, \alpha \rangle^{op}}(G, \overline{\alpha})} \ \ \mbox{and} \ \ \Rep_{\End_{\langle G, \alpha \rangle^{op}}(K, \beta)}$$ their 2-categories
of representations.

The canonical covariant functor
\begin{align*}
\Rep_{\langle G, \alpha \rangle^{op}} & \to \Rep_{\End_{\langle G, \alpha \rangle^{op}}(G, \overline{\alpha})^{op}} \\
(V, \chi) & \mapsto \Hom_{\langle G, \alpha \rangle^{op}}((G, \overline{\alpha}),(V,\chi))
\end{align*}
induces an equivalence of 2-categories from the category of right representations of the Segal-Mitchison $U(1)$-gerbe $\langle G, \alpha \rangle$ to the 2-category
of Lifting bundle $U(1)$-gerbes right representations of $\End_{\langle G, \alpha \rangle}(G, \overline{\alpha})$. The endomorphisms
act on the right by pre-composition of morphisms.

There is also the contravariant functor
\begin{align*}
\Rep_{\langle G, \alpha \rangle^{op}} & \to \Rep_{\End_{\langle G, \alpha \rangle}(K,\beta)} \\
(V, \chi) & \mapsto \Hom_{\langle G, \alpha \rangle^{op}}((V,\chi),(K, \beta))
\end{align*}
that takes a right representation of the Segal-Mitchison $U(1)$-gerbe $\langle G, \alpha \rangle$ to a left representation of the Pontrjagin dual group lifting bundle $U(1)$-gerbe $\End_{\langle G, \alpha \rangle}(K,\beta)$. The groupoid 
$\Hom_{\langle G, \alpha \rangle^{op}}((V,\chi),(K, \beta))$ is a lifting bundle $H^0(V,\underline{U(1)})^G$-gerbe and it becomes a left representation by composition of morphisms. Now, if we would want the image of the functor to produce only 
lifting bundle $U(1)$-gerbes, we would need to restrict to $G$-spaces $V$ which are homogeneous spaces themselves. In this way we would have
that $H^0(V,\underline{U(1)})^G=U(1)$.

The Lifting bundle $U(1)$-gerbe
$$\Hom_{\langle G, \alpha \rangle^{op}}((G, \overline{\alpha}),(K,\beta))$$
is a right representation of $\End_{\langle G, \alpha \rangle^{op}}(G, \overline{\alpha})$ and is a left representation of 
$\End_{\langle G, \alpha \rangle^{op}}(K, \beta)$. This lifting bundle $U(1)$-gerbe is the bi-representation that
permits to relate  $\End_{\langle G, \alpha \rangle^{op}}(G, \overline{\alpha})$ with $\End_{\langle G, \alpha \rangle^{op}}(K, \beta)$;
let us see more closely its structure.

\begin{proposition} \label{proposition birepresentation}
The representation $\Hom_{\langle G, \alpha \rangle^{op}}((G, \overline{\alpha}),(K,\beta))$ of  $\End_{\langle G, \alpha \rangle^{op}}(K, \beta)$
defines a left representation $(K, \widehat{\beta})$ of the multiplicative Segal-Mitchison $U(1)$-gerbe $\langle \widehat{ G}, \widehat{\alpha} \rangle$.
\end{proposition}
\begin{proof}
Recall from \eqref{eqn beta} and \eqref{eqn overline alpha} that 
\begin{align*}
\beta \in C^2(K \rtimes G, \underline{U(1)}),  & &  d_{K \rtimes G} \beta = \pi^*\alpha\\
\overline{\alpha} \in C^2(G \rtimes G, \underline{U(1)}), & &d_{G \rtimes G} \overline{\alpha}= \overline{\pi}^*\alpha
\end{align*}
 for $\pi : K \to *$ and $\overline{\pi} : G \to *$ respectively. Denote by $p: G \to K$ the projection homomorphism $g \mapsto [g]$
and recall from  \eqref{eqn map(K,K)} the notation $f_h \in \Map_{G^{op}}(K,K)$, $f_h(k)=hk$ for $h, k \in K$. Then we have the isomorphism
\begin{align*}
\Map_{G^{op}}(G,K) \cong K \ \ \ 
f_h \circ p  \mapsto h. 
\end{align*}

From the proof of Theorem \ref{Theorem crossed module}  and \eqref{Theorem crossed module} we know that there exists $\gamma$ such
$\beta - f_h^*\beta = (d_{K \rtimes G} \gamma)(h)$ for all $h \in K$, and therefore 
$$d_{G \rtimes G}( \overline{\alpha} - p^*f_h^* \beta )=0$$
for all $h \in K$. Since the complex $C^*(G \rtimes G, \underline{U(1)})$ is acyclic we see that the four term sequence
for the lifting bundle $U(1)$-gerbe $\Hom_{\langle G, \alpha \rangle^{op}}((G, \overline{\alpha}),(K,\beta))$ is simply
$$U(1) \To C^0(G \rtimes G, \underline{U(1)}) \circlearrowright \Hom^0_{\langle G, \alpha \rangle^{op}}((G, \overline{\alpha}),(K,\beta)) \To K.$$

The action of $\End_{\langle G, \alpha \rangle^{op}}(K, \beta)$ on $\Hom_{\langle G, \alpha \rangle^{op}}((G, \overline{\alpha}),(K,\beta))$ induces a left action of $\widehat{G}$ on $K$ by left multiplication. Moreover, this action implies that the Segal-Mitchison $U(1)$-gerbe
associated to $\Hom_{\langle G, \alpha \rangle^{op}}((G, \overline{\alpha}),(K,\beta))$ can be made into a left representation
of the multiplicative Segal-Mitchison $U(1)$-gerbe $\langle \widehat{ G}, \widehat{\alpha} \rangle$. Therefore the must exist
$$ \widehat{\beta} \in C^2(\widehat{G} \ltimes K, \underline{(U(1)}), \ \ \mbox{with} \ \ 
d_{\widehat{G} \ltimes K} \widehat{\beta} = \widehat{\pi}^* \widehat{\alpha}$$
where $\widehat{\pi} :\widehat{G} \to K$, $\widehat{g} \mapsto [\widehat{g}]$. 

Hence $(K, \widehat{\beta})$ is the desired left representation of $\langle \widehat{ G}, \widehat{\alpha} \rangle$ induced by 
$$\Hom_{\langle G, \alpha \rangle^{op}}((G, \overline{\alpha}),(K,\beta)).$$
\end{proof}

Now we will argue that the representation $(K, \widehat{\beta})$ satisfies the hypothesis of Theorem \ref{Theorem crossed module}
thus implying that $[\widehat{\alpha}] \in \Omega(\widehat{G},\widehat{S})$. For this to happen we just need to verify that
the projection map
$$
\End_{\langle \widehat{ G}, \widehat{\alpha} \rangle}(K, \widehat{\beta}) \to \Map_{\widehat{G}}(K,K)
$$
is surjective (see \eqref{eqn projection EndG(K,K)} and the argument of Theorem \ref{Theorem crossed module}).
This surjectivity is equivalent to the surjectivity of the projection map
$$
\End_{\End_{\langle G, \alpha \rangle^{op}}(K, \beta)} \left( \Hom_{\langle G, \alpha \rangle^{op}}((G, \overline{\alpha}),(K,\beta)) \right) \to
\End_{\End_{G^{op}}(K)} \left( \Hom_{G^{op}}(G, K) \right)
$$
induced by projecting on the components of the endomorphisms given by the maps among spaces. 

This surjectivity is part from a bigger result that relates the multiplicative Segal-Mitchison $U(1)$-gerbe
$\langle G , \alpha \rangle$ with its double dual constructed from the representation $(K, \beta)$. The setup is the following.

Since left and right compositions commute, the composition on the right by elements in $\End_{\langle G, \alpha \rangle^{op}}(G, \overline{\alpha})$
induce endomorphisms of $\Hom_{\langle G, \alpha \rangle^{op}}((G, \overline{\alpha}),(K,\beta))$ as left 
$\End_{\langle G, \alpha \rangle^{op}}(K, \beta)$ representations. Therefore we obtain the right composition functor
\begin{align}
\Psi: \End_{\langle G, \alpha \rangle^{op}}(G, \overline{\alpha}) & \to \End_{\End_{\langle G, \alpha \rangle^{op}}(K, \beta)} \left( \Hom_{\langle G, \alpha \rangle^{op}}((G, \overline{\alpha}),(K,\beta)) \right) \nonumber \\
\mathcal{F} & \mapsto \Psi(\mathcal{F}) \ \ ( \Psi(\mathcal{F}) : \mathcal{Q} \mapsto \mathcal{Q} \circ \mathcal{F}) \label{functor Psi}
\end{align}
with $\mathcal{T} \circ \Psi(\mathcal{F})(\mathcal{Q}) =  \mathcal{T} \circ  \mathcal{Q} \circ \mathcal{F}= \Psi(\mathcal{F})(\mathcal{T} \circ \mathcal{Q})$ for all $\mathcal{T} \in \End_{\langle G, \alpha \rangle^{op}}(K, \beta)$. We claim the following result.

\begin{theorem} \label{theorem double dual}
Let $\langle G , \alpha \rangle$ be a multiplicative Segal-Mitchison $U(1)$-gerbe with $[\alpha] \in \Omega(G,S)$ and $(K, \beta)$ a representation as in Theorem \ref{Theorem crossed module}. Then the canonical right composition functor
\begin{align*}
\Psi :\End_{\langle G, \alpha \rangle^{op}}(G, \overline{\alpha}) & \to \End_{\End_{\langle G, \alpha \rangle^{op}}(K, \beta)} \left( \Hom_{\langle G, \alpha \rangle^{op}}((G, \overline{\alpha}),(K,\beta)) \right) 
\end{align*}
from $\End_{\langle G, \alpha \rangle^{op}}(G, \overline{\alpha})$ to its double dual is an equivalence of 2-groups.
\end{theorem}
\begin{proof}
The outline of the proof is the following. First we will show that the projection map
\begin{align} \label{projection map}
\End_{\End_{\langle G, \alpha \rangle^{op}}(K, \beta)} \left( \Hom_{\langle G, \alpha \rangle^{op}}((G, \overline{\alpha}),(K,\beta)) \right) \to
\End_{\End_{G^{op}}(K)} \left( \Hom_{G^{op}}(G, K) \right)
\end{align}
is surjective. The surjectivity implies that $[\widehat{\alpha}] \in \Omega(\widehat{G}, \widehat{S})$ and by Theorem \ref{Theorem crossed module} we obtain that the isomorphism
classes of objects  $$\widehat{ \widehat{G}} = \pi_0\left(\End_{\End_{\langle G, \alpha \rangle^{op}}(K, \beta)} \left( \Hom_{\langle G, \alpha \rangle^{op}}((G, \overline{\alpha}),(K,\beta)) \right) \right)$$
 fits into the short exact sequence of groups
$\widehat{\widehat{S}} \To \widehat{\widehat{G}} \To K$.  This implies also that 
the  induced homomorphism of groups $\Psi: G \to \widehat{\widehat{G}}$ factors through the projection homomorphisms to $K$
$$
\xymatrix{
G \ar[rr]^\Psi \ar[dr]^p  && \ar[ld]_{\widehat{p}} \widehat{\widehat{G}} \\
&K.&
}
$$
Second we will show that the homomorphism, once restricted to the subgroup $S \subset G$, gives
the fiber $\widehat{\widehat{S}}$ via the Pontrjagin dual isomorphism $S \stackrel{\cong}{\to} \widehat{\widehat{S}}$. The two facts together imply that $\Psi$ induces an isomorphism of groups $G \cong \widehat{\widehat{G}}$, and together with the fact that $\Psi$ is a functor of group lifting bundle $U(1)$-gerbes, we have that $\Psi$ induces an equivalence of 2-groups.

Let us start with the surjectivity of the projection map of \eqref{projection map}. Let $(F_g, \eta_g) \in
\End^0_{\langle G, \alpha \rangle^{op}}(G, \overline{\alpha})$ with $F_g : G \to G $, $F_g(g') = gg'$ for all $g, g' \in G$,
and $\overline{\alpha} - F_g^*\overline{\alpha} = d_{G \rtimes G} \eta_g$. Every morphism in 
$\Hom^0_{\langle G, \alpha \rangle^{op}}((G, \overline{\alpha}),(K,\beta))  $ is of the form $(f_h \circ p, \epsilon)$
with $f_h : K \to K $, $f_h(k)=hk$ for all $h,k \in K$, $p:G \to K$ the projection homomorphism and
$\overline{\alpha} - p^* f_h^* \beta = d_{G \rtimes G} \epsilon$.

The right composition functor $\Psi$ of \eqref{functor Psi} applied to $(F_g, \eta_g)$ acts on $(f_h \circ p, \epsilon)$
by right composition as follows:
$$\left(\Psi(F_g, \eta_g) \right) (f_h \circ p, \epsilon) = (f_h \circ p \circ F_g, \eta_g +F_g^* \epsilon).$$

The composition map $f_h \circ p \circ F_g$ is equal to the map $f_{hp(g)} \circ p$. Therefore, the projection
map of \eqref{projection map} applied to $\Psi(F_g, \eta_g)$, is equivalent to the assignment
$f_h \circ p \mapsto f_{hp(g)} \circ p$ for all $ f_h \circ p \in \Hom_{G^{op}}(G,K)$ with $h \in K$. Since we know that
$$K \cong \End_{\End_{G^{op}}(K)} \left( \Hom_{G^{op}}(G, K) \right), \ \ l \mapsto ( f_h \circ p \mapsto f_{hl} \circ p)$$
we see that the projection map of \eqref{projection map} is surjective.

Now let us understand the image under $\Psi$ of  the endomorphism  functors $(F_s, \eta_s) \in \End^0_{\langle G, \alpha \rangle^{op}}(G, \overline{\alpha})$ whenever $s \in S \subset G$. Note that we can take a continuous homomorphism
$S \to \End^0_{\langle G, \alpha \rangle^{op}}(G, \overline{\alpha})$, $s \mapsto (F_s, \eta_s)$ since we know
that $[\alpha|_S]=0$. Hence $S$ acts on the short exact sequence
\begin{align}
U(1) \To C^0(G \rtimes G , \underline{U(1)}) \To Z^1(G \rtimes G , \underline{U(1)}) \label{short exact sequence GxG}
\end{align}
via the map $F_s^*$, and the relevant part of $\End_{\langle G, \alpha \rangle^{op}}(K, \beta)$ acts 
via the pullback of the complex
\begin{align}
U(1) \To C^0(K \rtimes G , \underline{U(1)}) \To Z^1(K \rtimes G , \underline{U(1)}); \label{complex of KxG}
\end{align}
both actions commute. The maps $F_s^*$ act on the complex $ C^*(G \rtimes G , \underline{U(1)}) $
via the pullback of the left translation action $F_s : G \rtimes G \to G \rtimes G$, $F_s(g,h)=(sg,h)$ for all $g,h \in G$.
The map $F_s^*$ provides an external automorphism of the short exact sequence of \eqref{short exact sequence GxG}
which commutes with the action of the complex in \eqref{complex of KxG}. Therefore the endomorphism
functors $\Psi(F_s, \eta_s)$, once projected to their isomorphism classes, generate the group $S \cong \widehat{\widehat{S}} \subset \widehat{\widehat{G}}$. Hence we have that the functor $\Psi$ induces an equivalence of groups in lifting bundle $U(1)$-gerbes as desired.
\end{proof}
\subsection{Morita equivalence}
Let us now restrict our representation category $\Rep_{\langle G , \alpha \rangle}$ to objects $(V, \chi)$ such that
$$H^0(V \rtimes G , \underline{U(1)})= U(1);$$
these are $G$-spaces $V$ on which the action of $G$ is transitive.
Denote this subcategory by $\Rep^+_{\langle G , \alpha \rangle}$. 
We carry out these restrictions to ensure that the category of morphisms
$$\Hom_{\langle G , \alpha \rangle^{op}}\left( (V , \chi),(K ,\beta) \right)$$
is a lifting bundle $U(1)$-gerbe for all $(V ,\chi) \in \Rep^+_{\langle G , \alpha \rangle}$.

Denote by
\begin{align*}
\CC : =& \End_{\langle G , \alpha \rangle^{op}}(G, \overline{\alpha}),\\
\DD := & \End_{\langle G , \alpha \rangle^{op}}(K,\beta),\\
\FF := & \Hom_{\langle G , \alpha \rangle^{op}}\left((G, \overline{\alpha})(K, \beta)\right),
\end{align*}
note that Theorem \ref{theorem double dual} claims that
$$\Psi : \CC \to \End_{\DD}(\FF)$$
is an equivalence of 2-groups, and 
 denote also by $\Rep^+$ the restriction of  the categories of representations of the endomorphism groups in lifting bundle gerbes
where the induced action of the groups $G$ and $\widehat{G}$ on the underlying spaces of the lifting bundle gerbes is transitive. 

Consider the covariant functor
\begin{align*} 
\Rep^+_{\langle G , \alpha \rangle^{op}} \to \Rep^+_{\CC^{op}}, \ \ (V,\chi) \mapsto \Hom_{\langle G , \alpha \rangle^{op}}\left((G,\overline{\alpha})(V,\chi)\right)
\end{align*} and  the contravariant functors 
\begin{align} 
\Rep^+_{\langle G , \alpha \rangle^{op}} \to& \Rep^+_\DD & \Rep^+_\DD  \to & \Rep^+_{\End_{\DD}(\FF)^{op}} \nonumber\\
(V,\chi) \mapsto &\Hom_{\langle G , \alpha \rangle^{op}}\left((V,\chi)(K, \beta)\right) & \MM \mapsto & \Hom_\DD(\MM,\FF)  \label{contravariant functors},
\end{align}
where $\Hom_\DD(\MM,\FF)$ denotes strict morphisms of lifting bundle gerbes. The covariant functor compared with the composition of the contravariant functors gives us the following assignments:
\begin{align*}
\Rep^+_{\langle G , \alpha \rangle^{op}} \to& \Rep^+_{\CC^{op}}  &&& \Rep^+_{\langle G , \alpha \rangle^{op}} \to& \Rep^+_{\End_{\DD}(\FF)^{op}}\\
(G, \overline{\alpha}) \mapsto & \CC &&& (G, \overline{\alpha}) \mapsto & \End_{\DD}(\FF)\\
(K,\beta) \mapsto &\FF &&& (K,\beta) \mapsto & \Hom_\DD(\DD,\FF).
\end{align*}
From the previous assignments we see that the composition of the contravariant functors provides and equivalence of 2-categories,
and therefore we obtain as  byproduct the desired Morita equivalence.

We have seen in Theorem \ref{Theorem cohomology class} that the group in lifting bundle $U(1)$-gerbes
$\End_{\langle G , \alpha \rangle^{op}}(K, \beta)$ induces a multiplicative Segal-Mitchison $U(1)$-gerbes $\langle \widehat{G}, \widehat{\alpha}_{(K,\beta)} \rangle$. Take the category of its left representations $\Rep_{\langle \widehat{G} , \widehat{\alpha}_{(K,\beta)} \rangle}$
and restrict it to objects $(W, \eta)$ such that
$$H^0(\widehat{G} \ltimes W , \underline{U(1)})= U(1).$$
Denote this subcategory by $\Rep^+_{\langle \widehat{G} , \widehat{\alpha}_{(K,\beta)} \rangle}$.

\begin{theorem}
Let $\langle G, \alpha \rangle$ and $\langle \widehat{G}, \widehat{\alpha}_{(K, \beta)} \rangle$
be Pontrjagin dual multiplicative Segal-Mitchison $U(1)$-gerbes over the topological groups $G$ and $\widehat{G}$. Then, the 2-category $\Rep^+_{\langle G , \alpha \rangle^{op}}$ of right representations 
of $\langle G, \alpha \rangle$ is equivalent to the 2-category $\Rep^+_{\langle \widehat{G} , \widehat{\alpha}_{(K,\beta)} \rangle}$ of left representations of $\langle \widehat{G}, \widehat{\alpha}_{(K, \beta)} \rangle$.
\end{theorem}
\begin{proof}
Our proof is based on the principles underlying a Morita context \cite[25A.15]{Rowen}.
Nevertheless,  we need to add extra equivalences of categories for the setup to work.
The relevant bi-representations were:
\begin{align*}
\Hom_{\langle G, \alpha \rangle^{op}}((G, \overline{\alpha}) ,(K,\beta) )\ \ \mbox{and} \ \ 
\Hom_{\langle \widehat{G} , \widehat{\alpha}_{(K,\beta)} \rangle}\left((\widehat{G}, \overline{\widehat{\alpha}}_{(K,\beta)}),(K ,\widehat{\beta}) \right)
\end{align*}
where the former is a $\End_{\langle G, \alpha \rangle^{op}}(K,\beta) - \End_{\langle G, \alpha \rangle^{op}}(G, \overline{\alpha})$ representation and the former is a
$\End_{\langle \widehat{G}, \widehat{\alpha}_{(K,\beta)} \rangle}(K,\widehat{\beta}) - \End_{\langle \widehat{G}, \widehat{\alpha}_{(K,\beta)} \rangle}(\widehat{G}, \overline{\widehat{\alpha}}_{(K,\beta)})$ representation.
The necessary equivalences of 2-groups are the following
\begin{align}
\End_{\langle G, \alpha \rangle^{op}}(K,\beta) & \simeq \End_{\langle \widehat{G}, \widehat{\alpha}_{(K,\beta)} \rangle}(\widehat{G}, \overline{\widehat{\alpha}}_{(K,\beta)}) \label{middle equivalence}\\
\End_{\langle G, \alpha \rangle^{op}}(G, \overline{\alpha}) & \simeq 
\End_{\langle \widehat{G}, \widehat{\alpha}_{(K,\beta)} \rangle}(K,\widehat{\beta}), \label{extremes equivalence}
\end{align}
where the first equivalence of  \eqref{middle equivalence} follows from Theorem \ref{Theorem cohomology class}
and the second equivalence of \eqref{extremes equivalence} follows from Theorem \ref{theorem double dual}.

The functors defined in \eqref{contravariant functors} together with the equivalence
of Theorem \ref{theorem double dual} provides the equivalence.
\end{proof}

\subsection{Equivalence of centers}

The endomorphism categories of \eqref{middle equivalence} and \eqref{extremes equivalence}
are strict tensor categories and the equivalence of 2-groups of Theorem \ref{theorem double dual}
implies that their centers are all tensor equivalent. Let us elaborate.

Let $\mathcal{C}$ be a monoidal category and let its center be
 $$\mathcal{Z}(\mathcal{C}): = \End_{\mathcal{C}- \mathcal{C}}(\mathcal{C}),$$
the endomorphisms of $\mathcal{C }$ which are compatible with the left and right monoidal structures.

\begin{theorem}
The centers of  $\End_{\langle G, \alpha \rangle^{op}}(G, \overline{\alpha})$ and $\End_{\langle G, \alpha \rangle^{op}}(K,\beta)$are equivalent:
$$  \mathcal{Z}(\End_{\langle G, \alpha \rangle^{op}}(G, \overline{\alpha})) \simeq
\mathcal{Z}(\End_{\langle G, \alpha \rangle^{op}}(K,\beta)).$$
\end{theorem}
\begin{proof}
Consider the following equivalences

$$\xymatrix@R09pt{
\mathcal{Z} \left(\End_{\langle G, \alpha \rangle^{op}}(G, \overline{\alpha}) \right) \ar[d]_\simeq
\\ \End_{\End_{\langle G, \alpha \rangle^{op}}(K, \beta) - \End_{\langle G, \alpha \rangle^{op}}(G, \overline{\alpha})} \left( \Hom_{\langle G, \alpha \rangle^{op}}((G, \overline{\alpha}),(K,\beta)) \right) \\
\mathcal{Z} \left(\End_{\langle G, \alpha \rangle^{op}}(K, \beta) \right). \ar[u]^\simeq
}$$
The uppper one is an equivalence  induced by the monoidal equivalence
$$\xymatrix@R09pt{
\End_{\End_{\langle G, \alpha \rangle^{op}}(G, \overline{\alpha})} \left(\End_{\langle G, \alpha \rangle^{op}}(G, \overline{\alpha}) \right) \ar[d]_{\simeq}^\Psi \\ 
\End_{\End_{\langle G, \alpha \rangle^{op}}(K, \beta)} \left( \Hom_{\langle G, \alpha \rangle^{op}}((G, \overline{\alpha}),(K,\beta)) \right)
}$$
defined
by $\Psi$ of Theorem \ref{theorem double dual}. The bottom is an equivalence which is induced by the following natural equivalences
$$\xymatrix@R09pt{
 \End_{\End_{\langle G, \alpha \rangle^{op}}(G, \overline{\alpha})^{op}} \left( \Hom_{\langle G, \alpha \rangle^{op}}((G, \overline{\alpha}),(K,\beta)) \right) \\
\End_{\End_{\langle G, \alpha \rangle^{op}}(K, \beta)} \left(\End_{\langle G, \alpha \rangle^{op}}(K, \beta) \right) \ar[u]^\simeq \\
\End_{\langle G, \alpha \rangle^{op}}(K, \beta). \ar[u]^\simeq
}$$
\end{proof}

From the previous Theorem we see that Pontrjagin dual multipliciative Segal-Mitchison gerbes will have equivalent centers. 
This result will implies interesting consequences on the topological invariants associated to Pontrjagin dual multiplicative Segal-Mitchison gerbes. We postpone its study to a forthcoming publication.

\section{Appendix A: Crossed modules and 2-groups} \label{Appendix}

Here we summarize the equivalent concepts of Crossed Module and 2-group, 
together with its morphisms and its classification.

\subsection{Crossed modules}
\label{sec:crossedmodules}

A topological crossed module $\phi: N \to E$ consists of topological groups $N$ and $E$, of a continuous homomorphism
$\phi$, and of a right action of $E$ on $N$ by group homomorphisms, denoted by $n^e$ for $n \in N$ and $e \in E$, such that
$\phi$ is $E$-equivariant
$$\phi(n^e) = e^{-1}\phi(n)e$$
and $\phi$ satisfies the Peiffer identity
$$n^{\phi(m)}=m^{-1}nm.$$

The kernel of $\phi$ is always abelian and the image of $\phi$ is normal in $E$. Therefore the crossed module
induces an exact sequence
$$A \to N \to E \to G$$
where $A = \ker(\phi)$, $G = \Coker(\phi)$, and there is an induced action of $G$ on $A$. The groups $A$ and $G$ 
are sometimes denoted by $\pi_1(N \stackrel{\phi}{\to} E)$ and $\pi_0(N \stackrel{\phi}{\to} E)$, respectively.

Whenever the group $N$ is abelian the Pfeiffer identity implies that the subgroup $\phi(N)$ of $E$ acts trivially on $N$. Therefore
a crossed module $\phi: N \to E$ with $N$ abelian is nothing else but a homomorphism of groups $\phi$ and an action
of $G$ on $N$ such that $\phi$ is $E$-equivariant. We will denote the crossed module by the homomorphism $\phi: N \to E$.

\subsection{2-groups}

A topological 2-group is a group object in topological groupoids, or equivalently
a groupoid object in topological groups. Let $\GG$ be the groupoid in topological groups
and denote by $\GG_1$ the group of morphisms, $\GG_0$ the group of objects, $\source, \target : \GG_1 \to \GG_0$
the source and target homomorphisms and $\ident : \GG_0 \to \GG_1$ the identity homomorphism. To a 2-group one may associate the groups $\pi_0(\GG)$ and $\pi_1(\GG)$, the first
being the isomorphism classes of objects of $\GG$ which becomes a group by the group operation of $\GG$, and
the second the group of automorphism of the identity object $1_{\GG_0 } \in \GG_0$. 
The group $\pi_1(\GG)$ is abelian
and is endowed with a right action of $\pi_0(\GG)$.
Strict morphisms of 2-groups are continuous functors that are group homomorphisms on the level of objects and of morphisms. A strict
morphism of 2-groups is an equivalence if it induces isomorphisms at the level of $\pi_1$ and $\pi_0$ \cite[Thm. 8.3.7 \& Cor. 8.3.8]{BaezLauda}.

\subsection{Equivalence of 2-groups and crossed modules}

2-groups and crossed modules are equivalent concepts. If we have a 2-group $\GG$ let $N= \Ker(\source) \subset \GG_1$,
$E=\GG_0$, $\phi=\target|_{N}$ and the right action of $E$ on $N$ is obtained by the equation $n^e:= \ident(e)^{-1}n\ident(e)$,
then $\target:N \to E$ is a crossed module. If we have a crossed module $\phi : N \to E$, define 
the action groupoid $[E/N]$ whose group of morphisms is 
$\GG_1 := E \times N$, whose group of objects is $\GG_0=E$, 
and the structural maps are $\source(e,n)=e$, $\target(e,n)=e\phi(n)$ and $\ident(e)=(e,1_N)$. The product structure
on $E \times N$ is given by the semidirect product, $(e,n) \cdot (e',n')=(ee', n^{e'}n)$.

Two extremal cases are worth highlighting, i.e. whenever one of the two groups $N$ or $E$ on the crossed module $N \to E$ is trivial. For $G$ a group, the crossed
module $\{0\} \to G$ induces the 2-group $G \twoarrows G$ where the source and the target maps are the identity; this
2-group is simply denoted by $G$, sometimes by $G[0]$. For $A$ an abelian group, the crossed module $A \to \{0\}$ induces the 2-group
$ A \twoarrows \{0\}$ which is sometimes denoted $A[1]$. A 2-group of the form $A[1] \times G$  defines
the crossed module $A \to G$ with trivial homomorphism and trivial action of $G$ on $A$.

\subsection{Morphisms of crossed modules}
A strict morphism of crossed modules consists of group homomorphisms $\psi_2: N \to \overline{N}$ and $\psi_1: E \to \overline{E}$ commuting
with $\phi$ and $\overline{\phi}$ and which respect the actions. A strict morphism is an equivalence of crossed modules if
the induced homomorphisms on $\pi_2$ and $\pi_1$ are isomorphisms \cite[Thm. 8.3.7 \& Cor. 8.3.8]{BaezLauda}.

Non-strict morphisms between crossed modules where
coined {\it{butterflies}} by Noohi  \cite[Def. 4.1.3]{AldrovandiNoohi} and they are the following. Take two crossed modules
$\phi: N \to E$ and $\overline{\phi}:\overline{N} \to \overline{E}$, then a butterfly from the first to the second consist
of a diagram 
$$
\xymatrix{
N \ar[rd]^{\kappa} \ar[dd]_{\phi} &&\overline{N} \ar[ld]_{\iota} \ar[dd]^{\overline{\phi}} \\
& P \ar[dl]_{\pi} \ar[dr]^{J} & \\
E && \overline{E}
}
$$
of topological groups and group homomorphisms, such that $\overline{N} \stackrel{\iota}{\to} P \stackrel{\pi}{\to} E$ is a group extension, and
$N \stackrel{\kappa}{\to} P \stackrel{J}{\to} \overline{E}$ is a complex. The maps must satisfy the equivariance conditions
$$\iota ( \overline{n}^{J(p)} ) = p^{-1}\iota(\overline{n}) p, \ \ \ \ \ \kappa ( n^{\pi(p)} )= p^{-1}\kappa(n)p$$
where $\overline{n} \in \overline{N}$, $p \in P$ and $n \in N$.

A morphism of butterflies $P$ and $P'$ (a 2-morphisms) consists of a group homomorphism $\psi: P \to P'$ such that the diagram
$$
\xymatrix{
N \ar[rd]\ar[dd]_{\phi} \ar[r]&P' \ar[rdd] \ar[ldd] &\overline{N} \ar[l] \ar[ld] \ar[dd]^{\overline{\phi}} \\
& P \ar[u] \ar[dl] \ar[dr] & \\
E && \overline{E}
}
$$
commutes and is compatible with all structural maps. 

A strict morphism between the crossed modules, consisting of homomorphisms $f_1:N \to \overline{N}$ and
$f_0: E \to \overline{E}$ commuting with $\phi$ and $\phi'$ and preserving the actions, does define a butterfly in
which $P=E \ltimes \overline{N}$, $\pi(e, \overline{n})=e$, $J(e, \overline{n})=f_0(e)\overline{\phi}(\overline{n})$,
$\iota(\overline{n})=(1,\overline{n})$ and $\kappa(n)=(\phi(n), f_1(n^{-1}))$. The product in $E \ltimes \overline{N}$ is
$$(e_1, \overline{n}_1) \cdot  (e_2, \overline{n}_2) =(e_1e_2, \overline{n}_1^{f_0(e_2)} \overline{n}_2).$$

\subsection{Classification of crossed modules} \label{subsection classification crossed modules}

Let us consider the category of crossed modules whose underlying groups are
locally contractible, compactly generated and paracompact. Then the isomorphism classes
of crossed modules in this setup are classified by a group $G$, an abelian group $A$,
which is moreover a $G$-module, and a Segal-Mitchison cohomology class in $H^3(G_\bullet, \underline{A})$.
This result follows from \cite[Prop. 2.3.4]{Breen-2}, cf.  \cite{Rousseau, Schommer-Pries}, when considering the isomorphism classes
of  crosses module extensions of the crossed module $G[0]$ by the crossed module $A[1]$.

Whenever the action of $G$ on $A$ is trivial, we may say that the crossed module
defines a multiplicative Segal-Mitchison gerbe with structure group $A$ in the sense of this article. 

From the short exact sequence $A \to E_{(G)}A \to B_{(G)} A$ defined in \eqref{short exact sequence A->E_GA -> B_GA},
we obtain the isomorphism $$H^3(G_\bullet, \underline{A})  \cong H^2(G_\bullet, \underline{B_{(G)}A})$$
since $E_{(G)}A = \Map(G,A)$ is $G$-soft. A class in $H^3(G_\bullet, \underline{A})$ defines a short exact sequence of groups
$$B_{(G)}A  \to W \to G,$$
which induces a crossed module whose four term exact sequence is
$$A \to E_{(G)}A  \to W \to G.$$
This procedure allow us to construct a crossed module extension of $G[0]$ by $A[1]$ for any element in $H^3(G_\bullet, \underline{A})$ in a canonical way.
 
\section{Appendix B: Lifting bundle gerbes} \label{Appendix B}
The bundle gerbes that appear in this work are of specific type and deserve some explanation. They were denoted {\it lifting bundle gerbes} by Murray \cite{Murray-1996} and they are constructed using principal bundles. For a comprehensive study of gerbes and its properties we refer to \cite{Waldorf-4} and the references therein.

Consider the category $\Topab$ of compactly generated and
locally contractible Hausdorff topological abelian groups and continuous homomorphisms. Let us fix a topological group $A$ in $\Topab$
which will be the structure group of the gerbes \cite{Brylinski}.

\subsection{Objects} 
\label{Objects lifting bundle gerbes}

A \emph{lifting bundle $A$-gerbe} is a pair $(N,E)$ consisting of a central extension $A \to N \to N/A$ and of a principal $N/A$-bundle $E$. A lifting bundle $A$-gerbe can be understood as  a representative for the problem to lift the structure group of  the principal bundle $E$ from $N/A$ to $N$. In this article, another perspective will be useful. 

Associated to any lifting bundle $A$-gerbe is a topological action groupoid $[E/N]$, where $N$ acts on $E$ through the projection $N \to N/A$. We will use the notation $[E/N]$ for lifting bundle $A$-gerbes below. Note that the isotropy group of each object is $A$. The base space of a lifting bundle $A$-gerbe $[E/N]$ is the base space $E/N$ of the principal bundle $E$; it is the space
of isomorphism classes of objects of the topological groupoid $[E/N]$.
 We will denote the lifting bundle $A$-gerbes also by the four term sequence
$$A \To N \circlearrowright E \To E/N\text{.}$$
Given the abelian group $A$, there are two canonical examples:
\begin{enumerate}
\item 
The \emph{canonical lifting $A$-bundle gerbe} is $[EBA/EA]$; here, the central extension is $A \to EA \to BA$, and the principal $BA$-bundle is $EBA \to B^2A$. As a  four term sequence it looks as follows:
$$A \To EA \circlearrowright EBA \To B^2A.$$ 

\item
The \emph{trivial lifting $A$-gerbe} is $[\ast/A]$, i.e., $N=A$ and $E=\ast$.  

\end{enumerate}

The Lifting bundle gerbes appearing in this work are ones with  $N$ abelian, which we then call abelian lifting $A$-gerbes. In most of the cases it is the degree 0 group $C^0$ of a complex, and the group $A$ is the zeroth cohomology group $H^0\subset C^0$. The space $E$
consists of  pairs of points on a fixed space together with elements in the degree 1 group $C^1$ of the complex. The action of $N$ on $E$ is simply
obtained by adding the image of the differential $d : C^0 \to C^1$.

\subsection{Morphisms}

Take two lifting bundle $A$-gerbes $[E/N]$ and $[D/M]$. A strict morphism is a continuous functor $\psi: [E/N]\to [D/M]$ that is $A$-equivariant at the level of morphisms. 

More explicitly, the functor $\psi$ is at the level of morphisms given by a continuous map
$\psi_1: E \times N \to D \times M$, and $A$ acts on both sides by multiplication on $N$ and $M$, respectively. The requirement is that $\psi_1$ is equivariant for these actions. 
Denoting for one moment the 1-category of lifting $A$-gerbes by $A$-$\Gerbe$ we have then:
\begin{multline*}
\Hom_{A\text{-}\Gerbe}([E/N], [D/M]) \\= \{\psi \in \Hom_{\Groupoid}([E/N], [D/M]) \mid \psi_1 \ \mbox{is} \ A\mbox{-equivariant}\}.
\end{multline*} 
It is important to notice that not all continuous functors are morphisms of  lifting$A$-gerbes. The most simple case on which this could be seen is the trivial gerbe $[*/A]$. On the one hand  $\Hom_{\Groupoid}([*/A],[*/A])=\Hom(A,A)$ consists off all endomorphism of groups of $A$. 
On the other hand $\Hom_{A\text{-}\Gerbe}([*/A],[*/A])$ consists of only the identity map. Notice also that the natural transformations
from the identity functor on $[A/*]$ to itself are parameterized by elements in $A$.

A special example of strict morphisms comes from principal bundle morphisms.
Let $E$ be a principal $N/A$-bundle, and let $D$ be a principal $M/A$-bundle, let $\phi: N \to M$ be an $A$-equivariant group homomorphism, and let $\psi_0: E \to D$ be a $\phi$-equivariant map, i.e., $\psi_0(e[n])=\psi(e)[\phi(n)]$. Then, $\psi_0$ and $\psi_1:=\psi_0 \times \phi$ from a strict morphism $\psi:[E/N] \to [D/M]$. Note that this works, in particular, when $M=N$ and $\phi=\id_N$.

\subsection{2-Morphisms}
Continuous natural transformations  provide the 2-morphisms that turn the 1-category of lifting bundle $A$-gerbes into a 2-category, which from now will be denoted by $A$-$\Gerbe$.

\subsection{Monoidal structure}  \label{monoidal structure A-gerbes}
The 2-category $A$-$\Gerbe$ of lifting bundle $A$-gerbes can be moreover be endowed with a monoidal structure. 
Take two lifting bundle $A$-gerbes $[E/N]$ and $[D/M]$ and define their product
 $$[E/N] \otimes  [D/M] : = [((E \times D) / (N \otimes M)]$$
where the group $N \otimes M$ is the quotient of $N \times M$ by its diagonal subgroup $\Delta:=\{(a,a^{-1}) \mid a \in A \}$. 
The trivial lifting $A$-gerbe $[\ast/A]$ is neutral with respect to this tensor product.

We remark that the underlying groupoid of $[E/N] \otimes  [D/M]$ is not the direct product of the groupoids $[E/N]$ and $[D/M]$, but that there is a continuous functor $[E/N] \times [D/M] \to [E/N] \otimes  [D/M]$ that is the identity on the level of objects. 
A continuous functor $\psi:[E/N] \times [E'/ N'] \to [D/M]$ induces 
a strict morphism $[E/N] \otimes [E'/ N'] \to [D/M]$ of  lifting bundle $A$-gerbes  if and only if its restriction to isotropy groups is the product  of the group $A$.
In terms of the trivial gerbe $[*/A]$, the only functor $[*/A] \times [*/A] \to [*/A]$ that induces
 a morphism of $A$-gerbes $[*/A] \otimes [*/A] \to [*/A]$ is the one induced by the product structure on $A$.

\subsection{Strict group  objects} \label{Group lifting bundle gerbes} A \emph{strictly multiplicative lifting bundle $A$-gerbe} is a strict group object in lifting bundle $A$-gerbes \cite{BaezLauda}. Thus, it is a lifting bundle $A$-gerbe $[E/N]$ together with  strict morphisms $[E/N] \otimes [E/N] \to [E/N]$ (multiplication), $[E/N] \to [E/N]$ (inversion), and $[\ast/A] \to [E/N]$ (unit), satisfying strictly the axioms of a group. Note that a strictly multiplicative lifting bundle $A$-gerbe
is in particular a group in the category of topological groupoids, i.e., a strict topological 2-group. In fact, these two conditions are equivalent:
\begin{lemma} \label{group object lifting bundle gerbe}
A lifting bundle $A$-gerbe
is strictly multiplicative if and only if the underlying groupoid is a strict topological 2-group. 
\end{lemma}

\begin{proof}
 The reason for this phenomenon
is the fact that for a groupoid $[*/A]$ to become a 2-group it is necessary that $A$ is abelian and that the product structure $[*/A] \times [*/A] \to [*/A]$
coincides with the product structure of $A$. Hence, the structural properties of groups in groupoids impose the conditions
necessary for the product to induce one on lifting bundle $A$-gerbes.
\end{proof}

We remark that for a strictly multiplicative lifting bundle $A$-gerbe, the spaces $E$ and $E/N$ are in fact topological groups, and that there is a group homomorphism $N \to E$. Hence we obtain an exact  four term sequence
$$A \To N \To E \To E/N\text{.}$$
Under the correspondence between strict 2-groups and crossed modules,  the above 4-term sequence is precisely the one associated to  a crossed module, see \S \ref{sec:crossedmodules}.

The trivial lifting bundle $A$-gerbe $[\ast/A]$ is strictly multiplicative in a unique way, and the canonical lifting bundle $A$-gerbe $[EBA/EA]$ is strictly multiplicative for the given group structure on $EBA$.
 Also, the 2-groups of endomorphisms
$\End_{\langle G, \alpha \rangle^{op}}(G, \overline{\alpha})$, $\End_{\langle G, \alpha \rangle}(K, \beta)$ and
$$\End_{\End_{\langle G, \alpha \rangle^{op}}(K, \beta)} \left( \Hom_{\langle G, \alpha \rangle^{op}}((G, \overline{\alpha}),(K,\beta)) \right) $$
defined in Example \ref{canonical representation}, Theorem \ref{Theorem crossed module} and Theorem \ref{theorem double dual}
respectively are strictly multiplicative  lifting bundle $U(1)$-gerbes.

\subsection{Representations of strictly multiplicative lifting bundle gerbes}

Strict group objects may act on objects. In our case, lifting bundle $A$-gerbes representations of group objects in lifting bundle $A$-gerbes  form a 2-category. This 2-category
is constructed in the usual way \cite{BaezLauda}. Objects are lifting bundle $A$-gerbes representations
of the specific group object in lifting bundle $A$-gerbes, morphisms are morphisms of lifting bundle $A$-gerbes compatible
with the action, and 2-morphisms are the natural transformations.

An example of a representation of a group in lifting bundle $U(1)$-gerbes is the lifting bundle $U(1)$-gerbe 
$\Hom_{\langle G, \alpha \rangle^{op}}((G, \overline{\alpha}),(K,\beta))$ of Proposition \ref{proposition birepresentation}. It is a right representation of the group lifting bundle $U(1)$-gerbe $\End_{\langle G, \alpha \rangle^{op}}(G, \overline{\alpha})$ and a left
representation of the  group lifting bundle $U(1)$-gerbe $\End_{\langle G, \alpha \rangle^{op}}(K, \beta)$.

\subsection{Weak morphisms}

Strict morphisms as defined above are not enough to capture the essence of the bicategory of lifting bundle
 $A$-gerbes (though, for the purpose of this article they suffice to discuss group objects).
There is a well established notion of weak morphisms between topological groupoids that correctly reflects the equivalence between topological stacks, known as anafunctors, bibundles, or Hilsum-Skandalis morphisms, see, e.g., \cite{Haefliger1984,Moerdijk1991,pronk}.  An $A$-equivariant (smooth) version has been introduced in \cite[App. A]{Waldorf-4}. Adapted to our setting, we obtain the following. A {\it weak morphism} between
lifting bundle $A$-gerbes $[E/N]$ and $[D/M]$ consists of the following structure:
\begin{itemize}
\item a space $F$  together with maps $\mu: F \to E$ and $\nu: F \to D$,
\item an action of $N$ on $F$ making $\mu$ $N$-equivariant,
\item an action of $M$ on $F$ making $\nu$ $M$-equivariant and $\mu$ a principal $M$-bundle,
\end{itemize}
such that the actions of $A \subset N$ and $A \subset M$ agree on $F$. Weak morphisms will be pictured as diagrams
$$
\xymatrix{
& F \ar[dl]_\mu \ar[rd]^\nu &\\
E && D.
}
$$
Whenever the map $\mu : F \to D$ is also an $N$-principal bundle, the weak morphism becomes a {\it weak isomorphism}, whose inverse is obtained by flipping the sides.
In this case we obtain homeomorphisms 
$$F/D \cong E, \ \ F/M \cong D, \ \ E/N \cong D/M$$ and the following diagram induce equivalence
of topological groupoids
$$
\xymatrix{
& [F/(M \times N)] \ar[dl]_\mu^\simeq \ar[rd]^\nu_\simeq &\\
[E/N] && [D/M].
}
$$

We shall see now that a strict morphism induces a weak one. Indeed, suppose $\Psi: [E/N] \to [D/M]$ is a strict morphism. We obtain a map $\psi : E/N \to D/M$ of the underlying spaces and we may consider the space $\Delta^*(E \times \psi^*D)$ where $\Delta : E/N \to E/N \times E/N$ is the diagonal map. There is a canonical action of $N \times M$ on
$\Delta^*(E \times \psi^*D)$ and canonical projection maps
$$\Delta^*(E \times \psi^*D) \to E \ \ \mbox{and} \ \ \Delta^*(E \times \psi^*D) \to D$$
where the first one is an $M$-principal bundle. The space $\Delta^*(E \times \psi^*D)$ together with the action and
the maps is the weak morphism induced by $\Psi$.

Weak morphisms are related through weak 2-morphisms, which are maps $F \to F'$ that exchange the maps to $E$ and $D$, and are equivariant with respect to the actions of $M$ and $N$. 
We refer the reader to the references above for a detailed study of the
2-categorical structure.
Now we are ready to show the following result.
\begin{theorem}  \label{theorem equivalence SM-gerbes LB-gerbes}
Let us consider the 2-category of Segal-Mitchison gerbes with structure group $A$ and the 2-category
lifting bundle $A$-gerbes with weak morphisms and weak 2-morphisms. Then,
the functor
\begin{align*}
A\text{-}\SMGrbs \to & A\text{-}\Gerbe\\
(M,\alpha)  \mapsto & [( \alpha^*EBA)/EA],
\end{align*}
is an equivalence of 2-categories.
\end{theorem}

\begin{proof} The functor is defined in section \S \ref{SM gerbes objects} at the level of objects,
in section \S \ref{SM gebres morphisms} at the level of morphisms and in section \S \ref{SM gebres 2-morphisms}
at the level of 2-morphisms.

 Let us now show that the functor is essentially surjective. That is, for a lifting bundle $A$-gerbe $[E/N]$ we need to find a Segal-Mitchison $A$-gerbe
$\alpha : E/N \to B^2A$ and a weak isomorphism between $[E/N]$ and $[\alpha^*EBA / EA]$. For this we need some preparation.

Since the principal bundle $N/A \to E \to E/N$ is obtained by pullback from the universal one $N/A \to E(N/A) \to B(N/A)$, we will
focus our attention in finding a canonical Segal-Mitchison $A$-gerbe 
$$\alpha : B(N/A) \To B^2A$$
and a weak isomorphism from $[E(N/A)/N]$ to $[\alpha^* EBA/EA]$.

Consider the diagonal action of $A$ on $N \times EA$ and denote $N \times_A EA$ the quotient group. Consider the diagram of topological groups
$$
\xymatrix{
N \ar[dr] & & EA \ar[ld]\\
& N \times _A EA \ar[ld]_{\pi_1} \ar[rd]^{\pi_2} & \\
N/A && BA
}
$$
where the diagonals are central extensions of groups, and $\pi_1$ and $\pi_2$ are the homomorphisms
induced by projections on the first and second coordinates, respectively. Applying the classifying space functor $B$, we obtain the diagram
$$
\xymatrix{
BN \ar[dr] & & BEA \ar[ld]\\
&B( N \times _A EA )\ar[ld]_{B\pi_1} \ar[rd]^{B\pi_2} & \\
B(N/A) \ar@/_1pc/[ur]_\sigma && B^2A
}
$$
 and note that there exists a section $\sigma : B(N/A) \to B( N \times _A EA )$ of the fiber map $B\pi_1$
since the fiber $BEA$ is contractible. Take $$\alpha : = B \pi_2 \circ \sigma : B(N/A) \to B^2A$$ and let us show
that the lifting bundle $A$-gerbes $[E(N/A)/N]$ and $[\alpha^*EBA /EA]$ are weakly isomorphic.

The diagram of principal bundles
$$\xymatrix{
E(N/A) \ar[d]^{p_1} & E(N \times_A EA) \ar[l]_-{E \pi_1} \ar[r]^-{E\pi_2} \ar[d]_{p}& EA \ar[d]_{p_2} \\
B(N/A) \ar@/_1pc/[r]_\sigma & B(N \times_A EA) \ar[l]_-{B \pi_1}  \ar[r]^-{B \pi_2} & B^2A
}$$
induces the diagram of principal bundles over the same base $B(N/A)$
$$\xymatrix{
E(N/A) \ar[d]^{p_1} & \sigma^*E(N \times_A EA) \ar[l]_-{E \pi_1} \ar[r]^-{E\pi_2} \ar[d]_{p}& \alpha^*EA \ar[d]_{p_2} \\
B(N/A) \ar@{=}[r] & B(N/A) \ar@{=}[r] & B(N/A).
}$$
The space $\sigma^*E(N \times_A EA)$ has free commuting actions of both $N$ and $EA$ and their actions agree on $A$.
The bundle $E \pi_2 : \sigma^*E(N \times_A EA) \to \alpha^* EA$ is an $N$-principal bundle and the bundle
$E \pi_1 : \sigma^*E(N \times_A EA) \to E(N/A)$ is an $EA$-principal bundle. Hence, we have a weak isomorphism.  
 
Note that whenever $N=EA$ we may take $\sigma : B^2A \to B(EA \times_A EA)$ as $\sigma := B \Delta$ where $\Delta : BA \to EA \times_A EA$, $\Delta[\lambda] = [(\lambda,\lambda)]$. The composition becomes
$$\alpha = B \pi_2 \circ B \Delta = B(\pi_2 \circ \Delta) = B(id) = id$$
and therefore the map $\alpha : B^2A \to B^2A$ is the identity.

Let us now show that the functor is fully faithful. Consider two Segal-Mitchison gerbes $(M, \alpha)$ and $(N,\beta)$,
and let $F$ be a weak morphism between $[\alpha^*EBA/EA]$ and $[\beta^*EBA/EA]$. We have then maps
$$
\xymatrix{
& F \ar[dl]_\mu \ar[rd]^\nu &\\
\alpha^*EBA  && \beta^*EBA
}
$$ where $F$ is endowed with an action of $EA \times EA$ and
$\mu$ is a principal $EA$-bundle. The maps $\mu$ and $\nu$ induce maps at the level of quotient
spaces 
$$
\xymatrix{
& F/(EA \times EA) \ar[dl]_{\widetilde{\mu}}^\cong \ar[rd]^{\widetilde{\nu}} &\\
M=\alpha^*EBA/EA  && \beta^*EBA/EA=N
}
$$
where $\widetilde{\mu}$ is a homeomorphism and $\widetilde{\mu}^{-1}$ is its inverse. 
It is now straightforward to check that the maps $\alpha$ and $\beta \circ \widetilde{\nu} \circ  \widetilde{\mu}^{-1}$
are homotopic, where $\widetilde{\nu} \circ  \widetilde{\mu}^{-1} : M \to N$. This implies
that the functor
$$\Hom_{A\text{-}\SMGrbs}((M, \alpha),(N, \beta)) \to \Hom^{\text{weak}}_{A\text{-}\Gerbe}([\alpha^*EBA/EA],[\beta^*EBA/EA])$$
is essentially surjective. In order to see that it is fully faithful, we recall that morphisms of Segal-Mitchison gerbes are mapped to \emph{strict} morphisms of lifting bundle $A$-gerbes. It is well known from above references about weak morphisms, that the functor
\begin{equation*}
\Hom^{\text{strict}}_{A\text{-}\Gerbe}([E/N],[D/M])\to \Hom^{\text{weak}}_{A\text{-}\Gerbe}([E/N],[D/M])
\end{equation*}
is fully faithful. Thus, it suffices to show that the 2-morphisms between Segal-Mitchison gerbes correspond one-to-one with 2-morphisms between strict morphisms of the corresponding lifting gerbes. This follows directly from the definitions.
\end{proof}

\section*{Conflict of interest}
On behalf of all authors, the corresponding author states that there is no conflict of interest.

    \bibliographystyle{alpha}
\bibliography{Bibliography-Gerbes}

   \end{document}